\documentclass{amsart}    
\usepackage{amsmath} 
\usepackage{paralist}
\usepackage{cancel} 
\usepackage{graphics} %% add this and next lines if pictures should be in esp format
\usepackage{epsfig} %For pictures: screened artwork should be set up with an 85 or 100 line screen
\usepackage{graphicx}  
\usepackage{epstopdf} %This is to transfer .eps figure to .pdf figure; please compile your paper using  PDFLaTex or PDFTeXify
\usepackage{mathrsfs}
\usepackage[colorlinks=true]{hyperref}
\hypersetup{urlcolor=blue, citecolor=red}

\newtheorem{Proposition}{Proposition}[section]

\newtheorem{Lemme}{Lemma}[section]
\newtheorem{Theoreme}{Theorem}[section]

\newtheorem{Corollaire}{Corollary}[section]

\def \vf{\vec{f}} 
\def \vef{\vec{\mathfrak{f}}}
\def \vg{\vec{g}}
\def \eg{\mathfrak{g}}
\def \vu{\vec{u}}
\def \vv{\vec{v}}
\def \vvg{\vec{{\bf g}}}
\def \veg{\vec{\mathfrak{g}}}
\def \P{\mathbb{P}}

\def \R{\mathbb{R}}

\def \Rt{\mathbb{R}^3}

\def \ds{\displaystyle}

\def \diver{\text{div}}
\usepackage{amssymb}

	\title[\bf Stationary Gravitational Boussinesq system ] %
{On Stationary Gevrey Solutions to the Gravitational  Boussinesq System and Applications to Uniqueness}  
\author[ Nestor Acevedo, Manuel Fernando Cortez  and Oscar Jarr\'in]{}

\subjclass[2020]{Primary: 35B65, 35B53; Secondary: 35B30} 
\keywords{Gravitation Boussinesq System; Stationary system; Gevrey class; Sobolev and H\"older regularity; Liouville-type Problem} 

\email{nestor.acevedo@udla.edu.ec}
\email{manuel.cortez@epn.edu.ec} 
\email{oscar.jarrin@udla.edu.ec}

\thanks{$^*$Corresponding author:  Oscar Jarr\'in}

\begin{document}
	\maketitle
	\begin{center}
	    \begin{minipage}{5cm}
        	% Enter the first author's name and address:
 	\centerline{\scshape Nestor Acevedo}
	\medskip
	{\footnotesize
	\centerline{Escuela de Ciencias Físicas y Matemáticas}
\centerline{Universidad de Las Américas}
\centerline{V\'ia a Nay\'on, C.P.170124, Quito, Ecuador}
	}
        \end{minipage}\hspace{2cm}
        \begin{minipage}{5cm}
        	% Enter the first author's name and address:
\centerline{\scshape Manuel Fernando Cortez}
\medskip
{\footnotesize
	\centerline{Departamento de Matem\'aticas}
	\centerline{Escuela Politécnica Nacional} 
	\centerline{Ladr\'on de Guevera E11-253, Quito, Ecuador} 
}
        \end{minipage}
	\end{center}

\medskip

\begin{center}
        \begin{minipage}{5cm}
	\centerline{\scshape Oscar Jarr\'in$^*$}
	\medskip
	{\footnotesize
		% Enter the address of the first author
		\centerline{Escuela de Ciencias Físicas y Matemáticas}
		\centerline{Universidad de Las Américas}
		\centerline{V\'ia a Nay\'on, C.P.170124, Quito, Ecuador}
	} 
\end{minipage}
\end{center}	
	
\bigskip
%%%%%%%%%%%%%%%%%%%%%%%%%%%%%%%%%%%%%%%%%%%%%%
\begin{abstract} The stationary version of the Boussinesq system with a general gravitational acceleration term is considered. Under suitable assumptions on this term, as well as on the external forces acting on each equation of this coupled system, we first establish the existence of weak solutions in the natural energy space $\dot{H}^1(\Rt)$. The uniqueness of these solutions is a challenging open problem.
	
	Within this framework, our first main contribution is to show that \emph{any} weak $\dot{H}^1$-solution exhibits an analytic smoothing effect in the Gevrey class. Our second main contribution is to show that the Gevrey class regularity can also be used to study the uniqueness problem, provided that these solutions satisfy a suitable low-frequency control.
	
	As a by-product, we also obtain new regularity results and a \emph{new Liouville-type result} for weak $\dot{H}^1$-solutions of the classical Navier--Stokes equations. 
\end{abstract}
%%%%%%%%%%%%%%%%%%%%%%%%%%%%%%%%%%%%%%%%%%%%%%
%%%%%%%%%%%%%%%%%%%%%%%%%%%%%%%%%%%%%%%%%%%%%%
{\footnotesize \tableofcontents}

\section{Introduction}
{\bf Setting.} In this work, we consider the stationary (time-independent) incompressible three-dimensional Boussinesq system posed in the entire space $\Rt$. Denote by $\vu : \Rt \to \Rt$ a solenoidal velocity field, by $P:\Rt \to \R$ the pressure, and by $\theta: \Rt \to \R$ the temperature of the fluid. The equations take the form
\begin{equation}\label{Boussinesq}
\begin{cases}\vspace{2mm}
-\Delta \vu + \text{div}(\vu \otimes \vu) + \vec{\nabla} P = \theta \vec{{\bf g}}+ \vf,  \qquad \text{div}(\vu)=0, \\
-\Delta \theta + \text{div}(\theta \vu) =  g,
\end{cases}  
\end{equation}
where $\vec{{\bf g}}: \Rt \to \Rt$ denotes the gravitational acceleration vector, while $\vf : \Rt \to \Rt$ and $g : \Rt \to \R$ represent external source terms in the velocity and temperature equations, respectively. For simplicity, and with no essential loss of generality, all physical constants have been normalized to one, and we assume that $\text{div}(\vf)=0$. 

\medskip

The system \eqref{Boussinesq} models the dynamics of a viscous incompressible fluid with thermal effects \cite{Boussinesq,Chandrasekhar}, while also incorporating the influence of a gravitational field $\vec{\bf g}$, such as the Earth’s gravity acting in oceanic and atmospheric flows \cite{Pedlosky}. Physically, this system arises as an approximation of a coupled model combining the classical Navier–Stokes equations with the laws of thermodynamics. In this framework, density variations induced by heat transfer are neglected in the continuity equation but retained in the momentum equation, where they appear as an additional buoyancy force proportional to both the temperature fluctuations and the gravitational acceleration. This mechanism explains the presence of the term $\theta \vvg$ in the first equation of this system \cite{Vallis}. 

\medskip

Mathematically, in the case when $\vf \equiv 0$, the steady-state Boussinesq system has been mainly studied in the framework of a \emph{bounded domain} $\Omega \subset \Rt$ with a \emph{smooth} boundary $\partial \Omega$, together with the Dirichlet boundary conditions:
\begin{equation}\label{Boussinesq-Domain}
\begin{cases}\vspace{2mm}
- \Delta \vu + \text{div}(\vu \otimes \vu) + \vec{\nabla} P = \theta \vec{{\bf g}}, \qquad \text{div}(\vu)=0, \quad  -\Delta \theta + \text{div}(\theta \vu) = g,  \quad  \text{in} \ \ \Omega, \\
\vu = \vu_b, \quad \theta=\theta_b \quad \mbox{on} \ \ \partial \Omega,
\end{cases}
\end{equation}
given by the prescribed functions $\vu_b: \partial \Omega \to \Rt$ and $\theta_b : \partial \Omega \to \R$. The qualitative properties of this system, such as the \emph{existence and regularity} of solutions, depend strongly on appropriate assumptions on the data $\vec{\bf g}$, $g$, $\vu_b$, and $\theta_b$, as well as on suitable conditions imposed on the boundary $\partial \Omega$.

\medskip

More precisely, in \cite{Avecedo}, considering a Lipschitz boundary $\partial \Omega$, and under the assumptions that 
\[ \vvg \in L^{\frac{3}{2}}(\Omega), \quad g \in H^{-1}(\Omega) \quad \text{and} \quad  \vu_b, \theta_b \in H^{\frac{1}{2}}(\partial\Omega),  \]
together with a (technical) smallness condition on $\vu_b$ across each connected component $\partial \Omega _i$ of the boundary $\partial \Omega$, it is proven that the system (\ref{Boussinesq-Domain}) has at least one weak solution 
\[ (\vu, \theta, P) \in  H^1(\Omega)\times  H^1(\Omega) \times L^2(\Omega).\]
Additionally, in the case of null Dirichlet boundary conditions when $\vu_b=\theta_b=0$, this solution satisfies the natural finite-energy estimates  $\ds{\| \vec{\nabla}\otimes \vu \|_{L^2} \leq C \| \vvg \|_{L^{\frac{3}{2}}}\, \| g \|_{H^{-1}}}$ and $\ds{\| \vec{\nabla}\theta \|_{L^2} \leq \| g \|_{H^{-1}}}$,  where the constant $C>0$ depends on the size of the domain $\Omega$. For additional related results, we refer to \cite{Gil,Kim,Maimoto1,Marimoto2} and the references therein.  

\medskip

Thereafter, by using the regularity theory of the Poisson and Stokes equations together with a well-designed iteration argument, a gain of $L^p$-regularity is also obtained in \cite{Avecedo}. More precisely, for a \emph{smoother} boundary $\partial \Omega$ of class $\mathcal{C}^{1,1}$, and for the parameters $p\geq \frac{6}{5}$ and $r>\frac{6}{5}$ (with $r$ technically related to $p$), assuming in addition that
\[
\vvg \in L^{r}(\Omega), \quad   g \in L^p(\Omega) \quad \text{and} \quad  \vu_b, \theta_b \in W^{2-\frac{1}{p},p}(\partial\Omega),
\]
it is proven that the \emph{particular solution obtained above} satisfies
\[
(\vu, \theta, P) \in W^{2,p}(\Omega)\times W^{2,p}(\Omega)\times W^{1,p}(\Omega).
\]

 It is important to emphasize that the results mentioned above, which were established for bounded domains with smooth boundaries, cannot be directly extended to the setting of $\Rt$, where the Boussinesq system (\ref{Boussinesq}) is posed. Indeed, several fundamental tools used in those analyses are no longer available in the whole-space framework, such as certain embedding properties of the $L^p(\Omega)$-spaces and the \emph{compact} Sobolev embeddings.
 
 \medskip
 
The analysis developed in this work relies on different methods and ideas. First, for the sake of completeness, we establish the existence of \emph{finite-energy weak solutions} in the natural homogeneous Sobolev space $\dot{H}^{1}(\Rt)$. The uniqueness of $\dot{H}^{1}$-solutions is a difficult and far from obvious open problem. Consequently, one of the main objectives of this article is to investigate the smoothing effect of \emph{any} $\dot{H}^1$-solution. Moreover, in contrast to \cite{Avecedo}, this smoothing effect is studied within a different framework, namely the Gevrey class. See expression (\ref{Def-Gevrey}) below for the corresponding definition.
  
\medskip

In the \emph{parabolic} framework, since the seminal work of C. Foias and R. Temam \cite{Foias-Temam}, the Gevrey class regularity of solutions has attracted considerable attention for a variety of \emph{time-dependent fluid models}. These include, among others, the classical Navier-Stokes equations, the Navier-Stokes-Voigt equations \cite{Kalantarov}, certain general dissipative equations \cite{Biswas}, visco-elastic second-grade fluid models \cite{Paicu}, the Newton-Boussinesq system \cite{Guo}, the Boussinesq boundary layer system \cite{Li}, and the classical Boussinesq system \cite{Zhou}.

\medskip

Despite the physical relevance and the mathematical complexity of these models, the contributions of these works can be broadly classified into two main directions. On the one hand, the works \cite{Biswas,Foias-Temam,Guo,Kalantarov,Zhou} show that, for initial data belonging to Sobolev spaces, the smoothing effects of the heat kernel (or related kernels) yield an instantaneous Gevrey class regularity for the corresponding solutions for later times $t>0$. On the other hand, the works \cite{Li,Paicu} establish that Gevrey class regularity imposed on the initial data is propagated by the evolution and, in certain situations, may even improve for the corresponding solutions.
 
\medskip

To the best of our knowledge, the Gevrey class regularity of the \emph{elliptic} (stationary) counterparts of these models has been much less explored, since most of the ideas used in the parabolic setting are no longer valid. For the classical Navier-Stokes equations, which are obtained from the Boussinesq system (\ref{Boussinesq}) when $\theta\equiv 0$:
\begin{equation}\label{Navier-Stokes}
-\Delta \vu + \text{div}(\vu \otimes \vu) + \vec{\nabla} P =  \vf,  \qquad \text{div}(\vu)=0,
\end{equation} 
we may mention the work \cite[Section $6$]{Kalantarov}, where the Gevrey class regularity is studied through the notion of the global attractor of the parabolic Navier-Stokes–Voigt equations with space-periodic boundary conditions. On the other hand, in the setting of the whole space $\Rt$, the work \cite{Chamorro-Jarrin-Lemarie-Gevrey} directly establishes the existence of Gevrey solutions to (\ref{Navier-Stokes}), provided that the external force $\vf$ satisfies suitable Gevrey class regularity assumptions.

\medskip

{\bf Main contributions}.  We show the persistence of Gevrey class regularity for \emph{any} $\dot{H}^{1}$-solution of the coupled system (\ref{Boussinesq}), provided that the data $\vf, g$ and $\vvg$ satisfy prescribed Gevrey class regularity assumptions. Moreover, we distinguish between the non-homogeneous case, when $\vf\neq 0$ and $g \neq 0$, and the homogeneous case, when $\vf=g=0$. Both cases reveal new information about the radius of analyticity of the solutions.

\medskip

As a by-product of the Gevrey class regularity, we also show that additional regularity assumptions on the data yield regularity properties for $\dot{H}^1$-solutions measured in different functional frameworks, such as the homogeneous Sobolev spaces $\dot{W}^{s,p}(\Rt)$ and H\"older spaces $\mathcal{C}^{s,\sigma}(\Rt)$, for suitable ranges of the parameters $s$, $p$, and $\sigma$. 
\medskip

Finally, for the homogeneous case of the Boussinesq system (\ref{Boussinesq}), we introduce new ideas to exploit the Gevrey class regularity of $\dot{H}^1$-solutions when studying their uniqueness, also known as a \emph{Liouville-type problem}. We therefore provide a new result ensuring that $\dot{H}^1$-solutions of this system satisfying a \emph{low-frequency control} vanish identically, that is, $\vu=0$ and $\theta=0$. This result is also of interest for the Navier-Stokes equations (\ref{Navier-Stokes}) when $\vf=0$. 

\medskip
 
{\bf Statement of the results and discussions}.  We begin by  establishing the existence of finite-energy weak $\dot{H}^1$-solutions. 
\begin{Proposition}\label{Prop:Existence-Solutions} 
	Let \(\vf\in \dot H^{-1}(\mathbb R^3)\) be such that \(\text{div} (\vf) = 0\), \(g\in \dot H^{-1}(\Rt)\), and \(\vvg\in  L^{\frac32}(\mathbb R^3)\cap \dot{H}^{\frac{1}{2}}(\Rt).\) Then the system (\ref{Boussinesq}) has at least one finite-energy weak solution 
	\[ \vu \in \dot H^1(\Rt), \quad \theta\in \dot H^1(\Rt) \quad \text{and}\quad P=P_{\vu}+P_{\theta}\in \dot{H}^{\frac{1}{2}}(\Rt)+ \dot H^1(\Rt).  \]
	Moreover, for a numerical constant $C>0$,  this solution satisfies the following energy estimate:
	\begin{equation}\label{Energy-Estimate}
	\|\vu\|^2_{\dot H^1} \leq C\, \left(  \|g\|^2_{\dot H^{-1}}\,\|\vvg\|^2_{L^{\frac32}}+\|\vf\|^2_{\dot {H^{-1}}}\right),  \quad   \|\theta\|^2_{\dot H^1} \leq C\|g\|^2_{\dot H^{-1}}, 
	\end{equation}
	and 
	\begin{equation*}
	\quad \| P_{\vu} \|_{\dot{H}^{\frac{1}{2}}} \leq C \| \vu \|^2_{\dot{H}^1}, \quad  \| P_\theta \|_{\dot{H}^1} \leq C\| \theta \|_{\dot{H}^1}\| \vvg \|_{\dot{H}^{\frac{1}{2}}}.
	\end{equation*}
	\end{Proposition}

Using the divergence-free property of $\vf$ and $\vu$, the pressure term $P$ can be characterized by 
\begin{equation}\label{Pressure}
P= (-\Delta)^{-1}\text{div} \left( \text{div}(\vu \otimes \vu)\right) - (-\Delta)^{-1} \text{div}(\theta \, \vvg):=P_{\vu}+P_{\theta},
\end{equation}
yielding that each term has different regularity properties, according to the regularity of $\vu$, $\theta$, and $\vvg$.

\medskip

The proof of this result is rather standard and follows from Schaefer's fixed-point argument, using arguments similar to those in \cite[Theorem $16.2$]{PGLR1} for the Navier-Stokes equation (\ref{Navier-Stokes}). Here, the main novelty lies in the different estimates required to handle the term $\theta\, \vvg$ in the first equation of (\ref{Boussinesq}). 

\medskip

Observe that this term mixes the data of the problem, given by the gravitational acceleration $\vvg$, with one of the unknowns of the system, namely the temperature $\theta$.  Therefore, we need the assumptions that $\vg \in L^{\frac{3}{2}}(\Rt)$ and $ \vvg \in  \dot{H}^{\frac{1}{2}}(\Rt)$. When comparing with the previous related work \cite{Avecedo}, the first assumption also naturally appears to obtain a natural energy estimate on the solutions, given in (\ref{Energy-Estimate}). On the other hand, the second assumption is essentially technical in order that some  local-compactness properties required in the  Schaefer's fixed-point argument to work. 

\medskip

It is worth mentioning that this  result  actually holds under the slightly more general assumption $\vvg \in L^3(\Rt)$ instead of $\vvg \in \dot{H}^{\frac{1}{2}}(\Rt)$. Nevertheless, this latter assumption provides a suitable framework when working with the Gevrey class.

\medskip

For a parameter $r >0$, we define the weighted exponential operator $e^{r\sqrt{-\Delta}}$ by the symbol $e^{r |\xi|}$. Thereafter, for $s \in \R$, we use the characterization introduced in \cite{Foias-Temam} to define the Gevrey class 
\begin{equation}\label{Def-Gevrey}
G^{s}_{r}(\Rt):= \big\{ \varphi \in \dot{H}^s(\Rt): \, e^{r \sqrt{-\Delta}} \varphi \in \dot{H}^{s}(\Rt)  \big\}.
\end{equation}
For $|s|<\frac{3}{2}$, this is a Banach space endowed with its natural norm   $\left\| e^{r\sqrt{-\Delta}} \varphi \right\|_{\dot{H}^s}$.  Moreover, for $s\geq 0$ the functions in $G^{s}_{r}(\Rt)$ are analytic, where $r$ measures the radius of analyticity.  

\medskip

Our first main contribution is devoted to showing the persistence of analytical smoothing effects for any $\dot{H}^1$-solution to the system (\ref{Boussinesq}). Specifically, we are interested in determining whether the \emph{given} radius of analyticity of the data $\vf, g$ and $\vvg$, denoted by $r$, increases or decreases compared with the \emph{obtained} radius of analyticity of $\dot{H}^1$-solutions, denoted by $\rho$.

\medskip

Additionally, we investigate whether $\rho$ depends only on the data or also on the solutions through their $\dot{H}^1$-norms. To this end, we first consider the non-homogeneous case of the system (\ref{Boussinesq}), where $\vf \neq 0$ and $g \neq 0$.

\begin{Theoreme}\label{Th:Gevrey} 
	Let $r>0$ and assume that the data in the non-homogeneous Boussinesq  system (\ref{Boussinesq}) satisfy
	\begin{equation}\label{Assumption-Forces-Gevrey}
	\vf \in G^{-1}_{r}(\Rt), 
	\quad 
	g \in G^{-1}_{r}(\Rt), 
	\quad 
	\text{and} 
	\quad  
	\vvg \in G^{\frac{1}{2}}_{r}(\Rt).
	\end{equation}
	Then there exists a parameter $0<\rho< \frac{2r}{3}$, depending only on $\| \vf\|_{G^{-1}_r}$, $\| g\|_{G^{-1}_r}$, $\|\vvg\|_{G^{\frac{1}{2}}_{r}}$  and $r$, such that for \emph{any} finite-energy weak solution $(\vu, \theta)\in \dot{H}^1(\Rt)$ of the system (\ref{Boussinesq}), associated with $\vf, g, \vvg$ and satisfying the energy control (\ref{Energy-Estimate}), it holds that
	\begin{equation*}
	\vu \in G^{1}_{\rho}(\Rt), 
	\quad 
	\theta \in G^{1}_{\rho}(\Rt),  
	\quad 
	\text{and} 
	\quad  
	P \in G^{\frac{1}{2}}_{\rho}(\Rt)+ G^1_\rho (\Rt).
	\end{equation*}
\end{Theoreme}

It is interesting to observe that $\rho$ is essentially smaller than $r$ by a factor of $\frac{2}{3}$ and, thanks to the energy control (\ref{Energy-Estimate}), this radius of analyticity depends only on the data and is uniform for any associated $\dot{H}^1$-solution. Of course, due to the identity (\ref{Pressure}), the pressure $P$ also exhibits an analytic smoothing effect.

\medskip

Let us briefly explain the general strategy of the proof. Our main idea is to view the stationary Boussinesq system (\ref{Boussinesq}) as a particular case of the time-dependent Boussinesq system:
\begin{equation}\label{Boussinesq-Evolution} 
\begin{cases}\vspace{2mm}
\partial_t \vv - \Delta \vv + \P\, \text{div}(\vv \otimes \vv) = \P (\vartheta\, \veg) +\P(\vef),  \quad \text{div}(\vv)=0, \quad  \partial_t \vartheta - \Delta\vartheta + \text{div}(\vartheta\, \vv)= \eg, \\ 
\vv(0,\cdot)=\vv_0, \quad \vartheta(0,\cdot)=\vartheta_0, 
\end{cases}
\end{equation} 
where, for $t>0$, $\vv(t,\cdot)$ and $\vartheta(t,\cdot)$ denote the solution, $\vv_0$ and $\vartheta_0$ denote generic $\dot{H}^1$ initial data, $\vef(t,\cdot)$, $\eg(t,\cdot)$ and $\veg(t,\cdot)$ are time-dependent external sources, and $\P$ stands for the well-known Leray projector.

\medskip

By exploiting the Fujita--Kato theory of mild solutions in the space $\mathcal{C}_t \dot{H}^{1}_x$, and by designing suitable external sources $\vef(t,\cdot)$, $\eg(t,\cdot)$ and $\veg(t,\cdot)$ from the given Gevrey data $\vf, g$ and $\vvg$ in Theorem \ref{Th:Gevrey}, we prove that the \emph{unique} local-in-time solution $(\vv, \vartheta)\in \mathcal{C}_t H^{1}_x$ arising from $\vv_0,\vartheta_0 \in \dot{H}^1(\Rt)$ exhibits an instantaneous Gevrey smoothing effect for $t>0$.

\medskip

On the other hand, observe that the stationary solution $(\vu, \theta) \in \dot{H}^1(\Rt)$ of the stationary system (\ref{Boussinesq}) satisfies $(\vu, \theta) \in \mathcal{C}_t \dot{H}^1_x$ and also solves the evolution system (\ref{Boussinesq-Evolution}) with initial data $(\vu, \theta)$. Therefore, by uniqueness we have $(\vv, \vartheta)=(\vu,\theta)$.

\medskip

This Gevrey class regularity of solutions also implies other smoothing effects, measured in other relevant functional spaces. We recall that, for $s\in \R$ and $1\leq p \leq +\infty$, the homogeneous Sobolev space $\dot{W}^{s,p}(\Rt)$ is defined by the norm $\| (-\Delta)^{\frac{s}{2}}(\cdot)\|_{L^p}$. On the other hand, for $s\geq 0$ and $0<\sigma<1$, the H\"older space $\mathcal{C}^{s,\sigma}(\Rt)$ consists of functions $\varphi$ whose fractional derivatives $(-\Delta)^{\frac{s}{2}}\varphi$ are H\"older continuous with parameter $\sigma$.

\begin{Corollaire}\label{Corollary:Holder} Under the same hypotheses as in Theorem \ref{Th:Gevrey}, let $k \in \mathbb{N}$ and assume in addition that the data in equation (\ref{Boussinesq}) satisfy:
	\begin{equation}\label{Assumption-W-k-infty}
	\vf, g, \vvg \in \dot{W}^{-1,\infty}(\Rt)\cap \dot{W}^{k,\infty}(\Rt).
	\end{equation}
	Then the following statements hold.
	\begin{enumerate}
		\item  For any $0\leq s \leq k+2$ and $6\leq p <+\infty$, we have
		\[\vu\in \dot{W}^{s,p}(\Rt), \quad \theta \in \dot{W}^{s,p}(\Rt) \quad \text{and} \quad   P \in \dot{W}^{s,p}(\Rt)+\dot{W}^{\min(s+1,k+1),p}(\Rt). \]
		
		\medskip 
		
		\item  Additionally, for $0\leq s \leq k+1$ and $\sigma:= 1-\frac{3}{p}>0$ (where $6\leq p <+\infty$), we have 
		\[  \vu \in  \mathcal{C}^{s,\sigma}(\Rt), \quad \theta\in  \mathcal{C}^{s,\sigma}(\Rt) \quad \text{and} \quad  P \in \mathcal{C}^{s,\sigma}(\Rt)+\mathcal{C}^{\min(s,k),\sigma}(\Rt). \]
	\end{enumerate}	
\end{Corollaire}

In (\ref{Assumption-W-k-infty}), the parameter $k$ quantifies the initial regularity of the data and determines the maximal gain of regularity for  solutions. More precisely, one can obtain a gain of order $k+2$ in the Sobolev space framework and of order $k+1$ in the H\"older space framework. This (expected) improvement of regularity arises from the regularizing effects of the Laplacian operator appearing in both equations of the system (\ref{Boussinesq}).

\medskip

Now, we turn our attention to the homogeneous case of the Boussinesq system (\ref{Boussinesq}), assuming from now on that $\vf=g=0$. For clarity of exposition, we rewrite the resulting system:
\begin{equation}\label{Boussinesq-Homogeneous}
\begin{cases}\vspace{2mm}
-\Delta \vu + \text{div}(\vu \otimes \vu) + \vec{\nabla} P = \theta \vec{{\bf g}},  \qquad \text{div}(\vu)=0, \\
-\Delta \theta + \text{div}(\theta \vu) = 0.
\end{cases}  
\end{equation}

In complete analogy with Theorem \ref{Th:Gevrey}, we study the analytical smoothing effects of $\dot{H}^1$-solutions to this system. Nevertheless, in contrast to our previous result, note that assuming the energy control (\ref{Energy-Estimate}) would render the analysis trivial, since we would only be considering the null solution $\vu=0$ and $\theta=0$. To overcome this issue, we remove the energy control (\ref{Energy-Estimate}), and we are able to state the following result for the homogeneous case.

\begin{Theoreme}\label{Th:Gevrey-Homogeneous} For $r>0$ assume that $\vvg \in G^{\frac{1}{2}}_{r}(\Rt)$. Then, for any weak solution $(\vu,\theta)\in \dot{H}^1(\Rt)$ of the homogeneous Boussinesq  system (\ref{Boussinesq-Homogeneous}) associated with $\vvg$, there exists a parameter $0<\varrho<\frac{2r}{3}$, which depends on $\| \vvg \|_{G^{\frac{1}{2}}_r}$, $\|\vu \|_{\dot{H}^1}$, $\| \theta \|_{\dot{H}^1}$ and $r$, such that 
		\begin{equation*}
	\vu \in G^{1}_{\varrho}(\Rt), \quad \theta \in G^{1}_{\varrho}(\Rt)  \quad \text{and}  \quad  P \in G^{\frac{1}{2}}_{\varrho}(\Rt)+ G^1_{\varrho} (\Rt).
	\end{equation*}
	\end{Theoreme}	

In contrast to Theorem \ref{Th:Gevrey}, the obtained radius of analyticity $\rho$ is no longer uniform and depends on the $\dot{H}^1$-norm of the solutions. For this fundamental reason, we have decided to state each theorem separately, even though the strategy of the proof is the same as presented above. Of course, Corollary \ref{Corollary:Holder} also holds for the system (\ref{Boussinesq-Homogeneous}).

\medskip

As already mentioned and to the best of our knowledge, the uniqueness of $\dot{H}^1$-solutions, the so-called Liouville-type problem, remains out of reach for both the coupled system (\ref{Boussinesq-Homogeneous}) and the stationary homogeneous Navier--Stokes equations (\ref{Navier-Stokes}) when $\vf=0$. 

\medskip

Briefly, multiplying the first equation in (\ref{Boussinesq-Homogeneous}) by $\vu$ and the second equation by $\theta$, and using the divergence-free property of $\vu$, an integration by parts \emph{formally} yields
\begin{equation}\label{Identities-Liouville}
\int_{\Rt} \text{div}(\vu \otimes \vu)\cdot \vu \, dx=0, 
\qquad 
\int_{\Rt} \text{div}(\theta \vu)\, \theta \, dx = 0,
\end{equation}
from which
\[
\int_{\Rt} |\vec{\nabla}\otimes \vu|^2\, dx =0 
\quad \text{and} \quad 
\int_{\Rt} |\vec{\nabla}\theta|^2 \, dx = 0,
\]
suggesting that $\vu, \theta \in \dot{H}^1(\Rt)$ must satisfy $\vu=0$ and $\theta = 0$. Nevertheless, the sole assumption that $\theta, \vu \in \dot{H}^1(\Rt)$ does not seem sufficient to rigorously justify the identities in (\ref{Identities-Liouville}),  and this fact makes the Liouville-type problem in the $\dot{H}^1$-space \emph{a difficult open problem}.
 
\medskip

For the particular case of the Navier--Stokes equations (\ref{Navier-Stokes}), there is a vast amount of literature devoted to the study of the Liouville-type problem in very different functional settings; see \emph{e.g.} \cite{Chae,Chamorro-Jarrin-Lemarie,Jarrin1,Jarrin2,Seregin} and the references therein. Most of these works \emph{do not include} the space $\dot{H}^1(\Rt)$ due to the difficulties mentioned above. In contrast, despite the large variety of these functional settings, the main idea behind them is that \emph{a priori decay properties of solutions in the space variable}, commonly characterized by Lebesgue, Lorentz, Morrey, and related spaces, imply the identity $\vu=0$ for smooth $\mathcal{C}^2$-solutions of (\ref{Navier-Stokes}).

\medskip

In this context, our second main contribution is to provide an application of the Gevrey smoothing effect to the study of the Liouville-type problem for the system $(\ref{Boussinesq-Homogeneous})$, and, as a by-product, for the equation $(\ref{Navier-Stokes})$, in the critical space $\dot{H}^1(\Rt)$.

\medskip

In our next result, differing from the previously mentioned works, we show that $\dot{H}^1$-solutions to (\ref{Boussinesq-Homogeneous}) vanish identically, provided that an \emph{a priori control on the frequency variable} holds.

 \begin{Theoreme}\label{Th:Liouville}
	Let $(\vu, \theta)\in \dot{H}^1(\Rt)$ be a weak solution of the homogeneous Boussinesq system (\ref{Boussinesq-Homogeneous}), associated with the gravitational acceleration $\vvg \in G^{\frac{1}{2}}_r(\Rt)$ for some $r>0$.
	
	\medskip
	
	For any $k \in \mathbb{N}$, define the dyadic annulus
	\begin{equation}\label{C_k}
		\mathscr{C}_k:= \left\{ \xi \in \Rt : 2^{-(k+1)}\leq |\xi| \leq 2^{-k} \right\}.
	\end{equation}
	Let $C>0$ be a generic constant independent of $k$. If it holds that
	\begin{equation}\label{Assumption-Liouville}
	\left\| \widehat{\vu} \right\|_{L^\infty(\mathscr{C}_k)}
	+
	\left\| \widehat{\theta} \right\|_{L^\infty(\mathscr{C}_k)}
	\leq C\, 2^k,
	\end{equation}
	then we have $\vu=0$ and $\theta=0$.
\end{Theoreme} 

We now briefly explain the idea of the proof in connection with the previous result. The hypothesis that $(\vu, \theta)\in \dot{H}^1(\Rt)$ and $\vvg \in G^{\frac{1}{2}}_r(\Rt)$ directly implies that $(\vu, \theta)\in G^1_\rho(\Rt)$ thanks to Theorem \ref{Th:Gevrey-Homogeneous}. This information, together with the \emph{low-frequency control} given in assumption (\ref{Assumption-Liouville}), allows us to derive sharp estimates showing that $(\vu, \theta)\in \dot{B}^{-1}_{\infty,\infty}(\Rt)$.

\medskip

 This homogeneous Besov space plays an important role in the theoretical study of the Navier--Stokes equations (both in the stationary and evolution cases), as it is the largest scale-invariant space for these equations; see \cite{Bourgain,PGLR1} for further references. 

\medskip 
 
Returning to our study, the fact that $(\vu, \theta)\in \dot{H}^1(\Rt)\cap \dot{B}^{-1}_{\infty,\infty}(\Rt)$ now justifies the identities in (\ref{Identities-Liouville}), yielding $\vu=0$ and $\theta=0$.  
 
\medskip

The \emph{frequency control} on $\dot{H}^1$-solutions seems to be a new idea for exploring the Liouville-type problem. To the best of our knowledge, this type of tool has only been used in the recent work \cite{Tan}. Among other interesting results, the Liouville-type problem is solved for $\dot{H}^1$-solutions of the Navier--Stokes equations satisfying 
\begin{equation}\label{Assumption-Liouville-Tan}
\liminf_{k\to -\infty} 2^{-k} \left\| \dot{S}_k \vu \right\|_{L^\infty(\Rt)}<+\infty,
\end{equation}
where, for any $k \in \mathbb{Z}$, the operator $\dot{S}_k$ denotes the standard low-frequency cut-off operator in the homogeneous Littlewood--Paley decomposition. See \cite[Theorem $1.1$ and Corollary $1.1$]{Tan} for details. 

\medskip

In the particular case when $\theta\equiv 0$ and $\vvg \equiv 0$, Theorem \ref{Th:Liouville} directly applies to the Navier--Stokes equations (\ref{Navier-Stokes}) provided that any weak $\dot{H}^1$-solution satisfies (\ref{Assumption-Liouville}). Comparing this assumption with (\ref{Assumption-Liouville-Tan}), the main difference is that in (\ref{Assumption-Liouville}) we consider the $L^\infty$-norm in the Fourier variable, while in (\ref{Assumption-Liouville-Tan}) this norm is taken in the space variable. Moreover, in (\ref{Assumption-Liouville-Tan}), when $k\to -\infty$ and consequently $2^{-k}\to +\infty$, this control essentially imposes a \emph{fast decay} of $\|\dot{S}_k \vu\|_{L^\infty(\Rt)}$. Conversely, in (\ref{Assumption-Liouville}) we allow the quantity $\|\widehat{\vu}\|_{L^\infty(\mathscr{C}_k)}$ to \emph{grow} like $2^k$ as $k\to +\infty$.

\medskip

{\bf Some conclusions and possible future research}.  As already mentioned, when $\theta\equiv 0$ and $\vvg \equiv 0$ all the results stated above hold in the particular framework of the stationary Navier--Stokes equation (\ref{Navier-Stokes}), providing new regularity criteria and a Liouville-type result for this equation.

\medskip

Within the setting of the Boussinesq system (\ref{Boussinesq}), we highlight that the gravitational acceleration vector $\vvg$ plays an important role in the present study through suitable decay and regularity assumptions on this term. It is therefore natural to ask what happens when this term is replaced by the fixed vertical vector $\vec{e}_3:=(0,0,1)$, which is also commonly used in Boussinesq models. To the best of our knowledge, the answer to this question is not trivial in the case of the whole space $\Rt$. 

\medskip

The ideas presented above to study the Gevrey smoothing effect of stationary solutions could be adapted to other relevant physical models in fluid dynamics. Some of them were mentioned in the Setting  section.  From a mathematical point of view, we believe that the visco-elastic second-grade fluid model studied in \cite{Paicu} is of particular interest, since, as highlighted in that work, the structure of this model presents weaker regularizing effects, leading to new difficulties.

\medskip

On the other hand, as pointed out in \cite{Tan}, it is also of interest to study new Liouville-type results under frequency assumptions such as (\ref{Assumption-Liouville}) and (\ref{Assumption-Liouville-Tan}) for the fractional version of the Navier--Stokes equations and related models.  

\medskip

{\bf Organization of the article}.   In Section \ref{Sec:Prelimilnaries} we summarize some well-known useful facts. Section \ref{Sec:Grevrey-non-homogenous} is devoted to the proof of Theorem \ref{Th:Gevrey} and Corollary \ref{Corollary:Holder}, while in  Section \ref{Sec:Homogeneous}  we give a short proof of Theorem \ref{Th:Gevrey-Homogeneous}. Finally, Section \ref{Sec:Liouville} is devoted to the proof of Theorem \ref{Th:Liouville}.

\section{Preliminaries}\label{Sec:Prelimilnaries}

In this section we summarize some well-known results that will be used in the sequel. We begin by stating some heat kernel estimates. For a proof, see \cite[Lemma $7.2$]{PGLR1}.

\begin{Lemme}[Heat kernel estimates]\label{Lem-Tech:Heat-Estimate}  
	The following estimates hold.
	\begin{enumerate}
		\item Let $\varphi \in L^{1}_{loc}([0,+\infty[, \dot{H}^1(\Rt))$. For any $t>0$, we have
		\begin{equation*}
		\left\| \int_{0}^{t} e^{(t-\tau)\Delta} \varphi(\tau,\cdot)d \tau\right\|_{\dot{H}^1} 
		\leq  C \int_{0}^{t}\| \varphi(\tau,\cdot)\|_{\dot{H}^1} d \tau.
		\end{equation*}
		
		\item Let $\varphi \in L^{2}_{loc}([0,+\infty[, L^2(\Rt))$. For any $t>0$, we have
		\begin{equation*}
		\left\| \int_{0}^{t} e^{(t-\tau)\Delta} \varphi(\tau,\cdot)d \tau\right\|_{\dot{H}^1} 
		\leq  C\, \left( \int_{0}^{t}  \| \varphi(\tau,\cdot)\|^{2}_{L^2} d \tau  \right)^{\frac{1}{2}}.
		\end{equation*}
	\end{enumerate}
\end{Lemme}

Next, we state the following fractional version of the Leibniz rule, known as the Kato--Ponce inequality. A proof can be found in \cite{Grafakos,Naibo}.

\begin{Lemme}[Fractional Leibniz rule]\label{Leibniz-rule} 
	Let $s>0$, $1<p<+\infty$ and $1<p_0, p_1, q_0, q_1 \leq +\infty$. Then there exists a constant $C>0$ such that
	\begin{equation*}
	\left\| (-\Delta)^{\frac{s}{2}}(\varphi_1\, \varphi_2) \right\|_{L^p}
	\leq 
	C\, \left\| (-\Delta)^{\frac{s}{2}} \varphi_1 \right\|_{L^{p_1}}\, \| \varphi_2 \|_{L^{p_2}}
	+
	C\, \| \varphi_1 \|_{L^{q_1}}\, \left\| (-\Delta)^{\frac{s}{2}} \varphi_2 \right\|_{L^{q_2}},
	\end{equation*}
	where $\frac{1}{p}=\frac{1}{p_1}+\frac{1}{p_2}=\frac{1}{q_1}+\frac{1}{q_2}$.
\end{Lemme}

Finally, we shall use the following result linking Morrey spaces and the Hölder regularity of functions. For a proof see \cite[Proposition $3.4$]{GigaMiyakawa}. Recall that for $1\leq p < +\infty$ the homogeneous Morrey space $\dot{M}^{1,p}(\Rt)$ is defined as the space of locally finite Borel measures $d\mu$ such that
\begin{equation}\label{Morrey}
\sup_{x_0 \in \Rt,\, R>0} 
R^{\frac{3}{p}} 
\left( 
\frac{1}{| B(x_0, R)|} 
\int_{B(x_0,R)} d| \mu | (x) 
\right)
<+\infty.
\end{equation}

\begin{Lemme}[H\"older regularity]\label{Lem-Holder-Reg} 
	Let $\varphi \in \mathcal{S}'(\Rt)$ such that $\vec{\nabla} \varphi \in \dot{M}^{1,p}(\Rt)$, with $p>3$. There exists a constant $C>0$ such that for all $x,y\in \Rt$ we have
	\[
	|\varphi(x)-\varphi(y)|
	\leq 
	C\, \| \vec{\nabla} \varphi \|_{\dot{M}^{1,p}} 
	|x-y|^{1-3/p}.
	\]
\end{Lemme}

\section{Existence of finite-energy weak solutions: proof of Proposition \ref{Prop:Existence-Solutions}}\label{Sec:Existece}  We consider the following approximate system. Let $\phi \in \mathcal{C}^{\infty}_{0}(\Rt)$ be a cut-off function satisfying $\phi(x)=1$ when $|x|\leq 1$, $\phi(x)=0$ when $|x|\geq 2$, and $0\leq\phi(x)\leq 1$ for all $x \in \Rt$. For $R>1$, we define $\phi_R(x)= \phi \left( \frac{x}{R}\right)$. 

Then, for $R>1$ consider the system:

\begin{equation}\label{BLLocalized}
\begin{cases}\vspace{2mm}
-\Delta \vu + \mathbb P\big((\phi_R\vu) \cdot\vec{\nabla} (\phi_R\vu)\big) = \mathbb P(\phi_R\theta \, \vec{{\bf g}})+ \vf,  \qquad \text{div}(\vu)=0, \\
-\Delta \theta + \phi_R\vu\cdot \vec{\nabla} (\phi_R\theta)  =  g.
\end{cases}  
\end{equation}

Observe that we have inserted \(\phi_R\) as a multiplicative factor for both \(\vu\) and \(\theta\). In addition, for the sake of simplicity, we have substituted the expressions \(\diver(\vu\otimes\vu)\) and \(\diver(\theta\vu)\) with \((\vu\cdot\vec{\nabla})\vu,\) and \(\vu\cdot \vec{\nabla}\theta\).

\medskip

Note that, as $R \to +\infty$, solutions of this system formally converge to a solution of the original Boussinesq system (\ref{Boussinesq}). In this framework, with $R>1$ fixed, we begin by constructing solutions to the system (\ref{BLLocalized}).

\medskip

We construct these solutions by means of Schaefer's fixed-point theorem, stated below. For a proof, see \cite[Theorem 16.1]{PGLR1}.

\begin{Theoreme}[Schaefer's fixed point]\label{Schaefer}
	Let $E$ be a Banach space and let $T: E \to E$ be an operator satisfying:
	\begin{enumerate}
		\item $T$ is continuous and compact.
		\item There exists a constant $M>0$ such that, for any parameter $0\leq \lambda \leq 1$, if $e=\lambda T(e)$, then $\| e \|_{E}\leq M$. 
	\end{enumerate}	
	Then, the fixed-point problem $e=T(e)$ admits at least one solution $e \in E$. 
\end{Theoreme}

In the setting of this theorem, we define the Banach space
\[
E= \left\{ (\vu,\theta) \in \dot{H}^1(\Rt) : \,  \text{div}(\vu)=0  \right\},
\]
equipped with its usual norm. 

\medskip

On the other hand, we rewrite the system (\ref{BLLocalized}) as the following fixed-point problem:
\begin{equation*}
\begin{cases}\vspace{2mm}
\vu= - (-\Delta)^{-1} \mathbb P\Big( ((\phi_R\vu) \cdot\vec{\nabla}) \phi_R\vu\Big)  + (-\Delta)^{-1} \mathbb P\big(\phi_R\theta \, \vec{{\bf g}}\big)+(-\Delta)^{-1}(\vf),  \\
\theta= -(-\Delta)^{-1}\big( \phi_R\vu\cdot \vec{\nabla} (\phi_R\theta)\big)+ (-\Delta)^{-1}(g),
\end{cases} 
\end{equation*}
where, from the right-hand side, we define the operator 
\begin{equation*}
 T_{R, \vvg} \left( \begin{array}{c}
\vu \\ 
 \theta
 \end{array}  \right):= \left(  \begin{array}{c}\vspace{2mm}
- (-\Delta)^{-1} \mathbb P \Big( ((\phi_R\vu) \cdot\vec{\nabla}) \phi_R\vu\Big)   + (-\Delta)^{-1} \mathbb P\big(\phi_R\theta  \,\vec{{\bf g}}\big)+(-\Delta)^{-1}(\vf)\\ 
-(-\Delta)^{-1}\big( \phi_R\vu\cdot \vec{\nabla} (\phi_R\theta)\big)+ (-\Delta)^{-1}(g)
\end{array} \right).
\end{equation*}

In the next technical lemmas, we verify that the operator $T_{R,\vvg}(\cdot)$ fulfills all the hypotheses stated in Theorem \ref{Schaefer}.

\begin{Lemme}\label{Lem-Tech-Continuity-Compactness} Let $R>1$ and $\vvg \in L^{\frac{3}{2}}(\Rt)\cap \dot{H}^{\frac{1}{2}}(\Rt)$. Then the operator $T_{R,\vvg}(\cdot):E \to E$ is continuous and compact. 
\end{Lemme}	
\begin{proof} We split the operator $T_{R,\vvg}(\cdot)$ as 
	\begin{equation*}
	T_{R,\vvg}\left( \begin{array}{c}
	\vu \\ 
	\theta
	\end{array}  \right)= \left(  \begin{array}{c}\vspace{2mm}
	- (-\Delta)^{-1} \mathbb P \Big( ((\phi_R\vu) \cdot\vec{\nabla}) \phi_R\vu\Big)   +(-\Delta)^{-1}(\vf)\\ 
	-(-\Delta)^{-1}\big( \phi_R\vu\cdot \vec{\nabla} (\phi_R\theta)\big)+ (-\Delta)^{-1}(g)
	\end{array} \right)+\left(  \begin{array}{c}\vspace{2mm}
	(-\Delta)^{-1} \mathbb P\big(\phi_R\theta  \,\vec{{\bf g}}\big)\\ 
	0
	\end{array} \right).
	\end{equation*}
	
	By \cite[Theorem $16.1$]{PGLR1} we know that the operator $- (-\Delta)^{-1} \mathbb P \Big( ((\phi_R\vu) \cdot\vec{\nabla}) \phi_R\vu\Big) +(-\Delta)^{-1}(\vf)$ is continuous and compact in the Banach space $\{ \vu \in \dot{H}^1(\Rt): \, \text{div}(\vu)=0 \}$. 
	
	\medskip
	
	Since the operator $-(-\Delta)^{-1}\big( \phi_R\vu\cdot \vec{\nabla} (\phi_R\theta)\big)+ (-\Delta)^{-1}(g)$ is structurally equal to the previous one ($\theta$ and $g$ have the same hypotheses as $\vu$ and $\vf$), we also have that this operator is continuous and compact in the space $E$ defined above. 
	
	\medskip
	
	Consequently, it is sufficient to prove that the operator $ (-\Delta)^{-1} \mathbb P\big(\phi_R\theta  \,\vec{{\bf g}}\big)$ is continuous and compact in $\dot{H}^1(\Rt)$. In the sequel, we will use a generic constant $C>0$, which depends on $R$ and may change from one line to the next. 
	
	\medskip 
	
	\emph{Continuity}. For any $\theta \in \dot{H}^1(\Rt)$, using the continuous embeddings $L^{\frac{6}{5}}(\Rt)\subset \dot{H}^{-1}(\Rt)$, $\dot{H}^1(\Rt)\subset L^6(\Rt)$ and H\"older inequalities (with $\frac{5}{6}=\frac{1}{6}+\frac{2}{3}$), we directly obtain
	\begin{equation*}
	\begin{split}
	\left\| (-\Delta)^{-1} \mathbb P\big(\phi_R\theta  \,\vec{{\bf g}}\big) \right\|_{\dot{H}^{1}} \leq &\, C \| \phi_R\theta  \,\vec{{\bf g}} \|_{\dot{H}^{-1}} \leq C\, \| \phi_R\theta  \,\vec{{\bf g}} \|_{L^{\frac{6}{5}}} \leq C\| \phi_R \|_{L^\infty}\, \| \theta  \,\vec{{\bf g}} \|_{L^{\frac{6}{5}}} \\
	\leq &\, C \| \theta \|_{L^6}\, \| \vvg \|_{L^{\frac{3}{2}}} \leq  C \| \theta \|_{\dot{H}^1}\, \| \vvg \|_{L^{\frac{3}{2}}}. 
	\end{split}
	\end{equation*}
	
	\emph{Compactness}. Let $(\theta_n)_{n\in \mathbb{N}}$ be a sequence in $\dot{H}^1(\Rt)$ which, for any $n \in \mathbb{N}$, satisfies
	\begin{equation}\label{Sequence-Bound}
	\| \theta_n \|_{\dot{H}^1}\leq K,
	\end{equation} 
	for a constant $K>0$. We will prove that there exists a subsequence $(\theta_{n_k})_{k \in \mathbb{N}}$ such that $ (-\Delta)^{-1} \mathbb P\big(\phi_R\theta_{n_k}  \,\vec{{\bf g}}\big)$ converges in the strong topology of $\dot{H}^{1}(\Rt)$. 
	
	\medskip
	
	First, we prove that the sequence $(\phi_R \theta_n)_{n\in \mathbb{N}}$ is also uniformly bounded in $\dot{H}^{1}(\Rt)$. Using H\"older inequalities (with $\frac{1}{2}=\frac{1}{3}+\frac{1}{6}$) and (\ref{Sequence-Bound}), for any $n \in \mathbb{N}$ we write
	\begin{equation}\label{Sequence-Bound-R}
	\begin{split}
	\| \phi_R \theta_n \|_{\dot{H}^1}\leq  &\, C \| \vec{\nabla}(\phi_R \theta_n)\|_{L^2}  \leq C \| (\vec{\nabla} \phi_R) \theta_n \|_{L^2}+ \| \phi_R \vec{\nabla} \theta_n\|_{L^2}\\
	\leq &\, C\, \| \vec{\nabla} \phi_R \|_{L^3}\|\theta_n\|_{L^6}+ C\, \|\phi_R \|_{L^\infty}\| \vec{\nabla} \theta_n \|_{L^2} \leq C\, \| \theta_n \|_{\dot{H}^1}\leq CK. 
	\end{split}
	\end{equation} 
	
	On the other hand, let $B_{8R}:= \{ x \in \Rt : \, |x|<8R \}$. We prove that there exists a subsequence such that $(\phi_R \theta_{n_k})_{k\in \mathbb{N}}$ converges in the strong topology of $L^p(B_{8R})$, for any $1\leq p < 6$.
	
	\medskip
	
	Indeed, by definition of the cut-off function $\phi_R$, for any $n \in \mathbb{N}$ we have $\text{supp}(\phi_R \theta_n)\subset B_{8R}$. Then, the desired convergence follows from the uniform bound proved in (\ref{Sequence-Bound-R}) and the Rellich–Kondrashov theorem (see \cite[Theorem $IX.16$]{Brezis}) in the Sobolev space $\dot{H}^{1}(B_{8R})$.  
	
	\medskip
	
	With this convergence property at hand, the convergence of $ (-\Delta)^{-1} \mathbb P\big(\phi_R\theta_{n_k}  \,\vec{{\bf g}}\big)$ follows from the following inequality, where we also use H\"older inequalities (with $\frac{5}{6}=\frac{1}{2}+\frac{1}{3}$) and the continuous embedding $\dot{H}^{\frac{1}{2}}(\Rt)\subset L^3(\Rt)$ :
	\begin{equation*}
	\left\| (-\Delta)^{-1} \mathbb P\big(\phi_R\theta_{n_k}  \,\vec{{\bf g}}\big) \right\|_{\dot{H}^1} \leq C\, \| \phi_R\theta_{n_k}\, \vvg \|_{\dot{H}^{-1}}\leq C\, \| \phi_R\theta_{n_k}\, \vvg\|_{L^{\frac{6}{5}}} \leq C\, \| \phi_R \theta_{n_k} \|_{L^2}\, \| \vvg \|_{L^3} \leq C\, \| \phi_R \theta_{n_k} \|_{L^2}\, \| \vvg \|_{\dot{H}^{\frac{1}{2}}}. 
	\end{equation*}
	
	Thus, we conclude the compactness of the operator $ (-\Delta)^{-1} \mathbb P\big(\phi_R\theta  \,\vec{{\bf g}}\big)$, and Lemma \ref{Lem-Tech-Continuity-Compactness} is now proven. 
\end{proof}

\begin{Lemme}\label{Lem-Tech-A-Priori} For any parameter $0<\lambda\leq 1$, let $(\vu, \theta) \in E$ be such that 
	\begin{equation*}
	\left( \begin{array}{c}
	\vu \\ 
	\theta
	\end{array}  \right) = 	\lambda T_{R, \vvg} \left( \begin{array}{c}
	\vu \\ 
	\theta
	\end{array}  \right).
	\end{equation*}
	Then, for a numerical constant $C>0$, it holds that
	\begin{equation}\label{Estim-Energy-A-Priori}
	\| \vu \|_{\dot{H}^1} \leq C\left( \| g \|_{\dot{H}^{-1}}\|\vvg \|_{L^{\frac{3}{2}}} + \| \vf \|_{\dot{H}^{-1}}  \right) \quad \text{and} \quad   \| \theta \|_{\dot{H}^1}\leq C \| g \|_{\dot{H}^1}.
	\end{equation}
\end{Lemme}	
\begin{proof} From the identity above, it follows that $(\vu, \theta)$ satisfies the system 
	\begin{equation*}
	\begin{cases}\vspace{2mm}
	-\Delta \vu = -\lambda \mathbb P\big((\phi_R\vu) \cdot\vec{\nabla} (\phi_R\vu)\big) + \lambda \mathbb P(\phi_R\theta \, \vec{{\bf g}})+ \lambda \vf,  \qquad \text{div}(\vu)=0, \\
	-\Delta \theta =-\lambda \phi_R\vu\cdot \vec{\nabla} (\phi_R\theta)  + \lambda  g.
	\end{cases}  
	\end{equation*}	
	As $(\vu, \theta)\in \dot{H}^{1}(\Rt)$, a standard energy estimate yields
	\begin{equation}\label{Energy1}
	\| \vu \|^2_{\dot{H}^2} = \lambda \int_{\Rt} \phi_R \theta \vvg \cdot \vu \, dx + \lambda \langle \vf, \vu \rangle_{\dot{H}^{-1}\times \dot{H}^1},
	\end{equation}
	and 
	\begin{equation}\label{Energy2}
	\| \theta \|^2_{\dot{H}^2}= \lambda \langle  g , \theta \rangle_{\dot{H}^{-1}\times \dot{H}^1},
	\end{equation}
	where, from the divergence-free property of $\vu$, we have used the identities 
	\[ \int_{\Rt} \mathbb P\big((\phi_R\vu) \cdot\vec{\nabla} (\phi_R\vu)\big)\cdot \vu \, dx =0 \quad \text{and}\quad \int_{\Rt}\phi_R\vu\cdot \vec{\nabla} (\phi_R\theta)  \theta \, dx =0.  \]
	
	From identity (\ref{Energy2}), as $0<\lambda\leq 1$ and by the $\dot{H}^{-1}-\dot{H}^1$ duality, we directly obtain 
	\[ \| \theta \|_{\dot{H}^1}\leq C\| g \|_{\dot{H}^1}. \]
	Thereafter, from identity (\ref{Energy1}), using the $\dot{H}^{-1}-\dot{H}^1$ duality together with H\"older inequalities (with $1=\frac{1}{6}+\frac{5}{6}$ and $\frac{5}{6}=\frac{1}{6}+\frac{2}{3}$), we write
	\begin{equation*}
	\begin{split}
	\| \vu \|^2_{\dot{H}^1} \leq &\,  \lambda \| \phi_R \theta \vvg \|_{L^{\frac{6}{5}}}\, \| \vu \|_{L^6}+C \| \vf \|_{\dot{H}^{-1}}\, \| \vu \|_{\dot{H}^1}   \leq  \lambda \| \phi_R \|_{L^\infty} \| \theta\|_{L^6}\| \vvg \|_{L^{\frac{3}{2}}}\, \| \vu \|_{L^6}+ \| \vf \|_{\dot{H}^{-1}}\, \| \vu \|_{\dot{H}^1} \\
	\leq &\, C\, \|\theta \|_{\dot{H}^1}\| \vvg \|_{L^{\frac{3}{2}}}\, \| \vu \|_{\dot{H}^1}+  C \| \vf \|_{\dot{H}^{-1}}\, \| \vu \|_{\dot{H}^1} \leq C \| g \|_{\dot{H}^{-1}}\|\vvg \|_{L^{\frac{3}{2}}}\| \vu \|_{\dot{H}^1}+ C \| \vf \|_{\dot{H}^{-1}}\, \| \vu \|_{\dot{H}^1},
	\end{split}
	\end{equation*}
	hence, we obtain
	\[ \| \vu \|_{\dot{H}^1} \leq C\left( \| g \|_{\dot{H}^{-1}}\|\vvg \|_{L^{\frac{3}{2}}} + \| \vf \|_{\dot{H}^{-1}}  \right). \]
\end{proof}

 End of the proof of Proposition \ref{Prop:Existence-Solutions}.  With Lemmas \ref{Lem-Tech-Continuity-Compactness} and \ref{Lem-Tech-A-Priori} at hand, a direct application of Theorem \ref{Schaefer} yields the existence of a solution $(\vu,\theta) \in E$ to the equation $	\left( \begin{array}{c}
\vu \\ 
\theta
\end{array}  \right) = 	T_{R, \vvg} \left( \begin{array}{c}
\vu \\ 
\theta
\end{array}  \right)$. 

\medskip

Therefore, for fixed $R>1$, it follows that the system (\ref{BLLocalized}) admits a solution $(\vu_{R}, \theta_R) \in \dot{H}^1(\Rt)$. Moreover, from (\ref{Estim-Energy-A-Priori}), the family of solutions $(\vu_R, \theta_R)_{R>1}$ is uniformly bounded in $\dot{H}^1(\Rt)$.

\medskip

Using this uniform bound and applying a standard passing-to-the-limit argument as $R\to +\infty$ (see, for instance, \cite[Theorem $4$]{Cortez-Jarrin} and \cite[Theorem $16.2$]{PGLR1}), we obtain a limit $(\vu,\theta)\in \dot{H}^1(\Rt)$. Moreover, using well-known properties of the Leray projector (see \cite[Lemma $6.3$]{PGLR1}), there exists $P \in \mathcal{D}'(\Rt)$ such that $(\vu,\theta,P)$ verifies the Boussinesq system (\ref{Boussinesq}) in the sense of distributions.   

\medskip

Finally,  we study the pressure $P$ given in (\ref{Pressure}).  As $\vu \in \dot{H}^1(\Rt)$, by product laws in homogeneous Sobolev spaces we have that $\vu \otimes \vu \in \dot{H}^{\frac{1}{2}}$; therefore, we obtain $ \| P_{\vu} \|_{\dot{H}^{\frac{1}{2}}} \leq C \| \vu \|^{2}_{\dot{H}^1}$. On the other hand, as $\theta \in \dot{H}^{1}(\Rt)\subset L^6(\Rt)$ and $\vvg \in L^3(\Rt)$, from H\"older inequalities we have $\theta \vvg \in L^{2}(\Rt)$; hence, $ \|P_{\theta}  \|_{\dot{H}^1} \leq C \| \theta\|_{\dot{H}^1}\| \vvg \|_{L^3} \leq C \, \| \theta\|_{\dot{H}^1}\| \vvg \|_{\dot{H}^{\frac{1}{2}}}$. 

\medskip 

Proposition \ref{Prop:Existence-Solutions} is now proven.

\section{Analyticity of weak $\dot{H}^1$-solutions in the non-homogeneous case}\label{Sec:Grevrey-non-homogenous}
\subsection{Proof of Theorem \ref{Th:Gevrey}}
The proof is divided into the following steps. 

\medskip

{\bf Step 1}. \emph{Fujita-Kato mild solutions for the evolution problem}.  We consider the evolution problem for the Boussinesq system

\medskip

Within this framework, we begin by constructing local-in-time solutions in the space $\mathcal{C}_{t}\dot{H}^{1}_x$. The proof of this result is standard and essentially follows the same arguments as in the case of the classical Navier–Stokes equation (when $\vartheta\equiv 0$). See \cite[Theorem 7.1]{PGLR1}.  Nevertheless, for the reader's convenience, we provide a sketch of the proof, including the estimates on $\P (\vartheta\, \veg)$.

\begin{Proposition}\label{Prop-Tech:Fujita-Kato} Let $\vv_0, \vartheta_0 \in \dot{H}^1(\Rt)$, where $\text{div}(\vv)=0$, $  \vef, \eg \in \mathcal{C}\big( [0,1] \dot{H}^{1}(\Rt)$ and $\veg \in \mathcal{C}\big([0,1], \dot{H}^{\frac{1}{2}}(\Rt)\big)$. Define the quantities
	\begin{equation}\label{delta}
	\delta_0:= C \left( \| \vv_0\|_{\dot{H}^1}+\| \vartheta_0\|_{\dot{H}^1_x} + \| \vef \|_{L^\infty_t\dot{H}^1_x}+ \| \eg \|_{L^{\infty}_t\dot{H}^1_x}\right) \quad \text{and}\quad \eta_0:= C \| \veg \|_{L^\infty_t \dot{H}^{\frac{1}{2}}_{x}},
	\end{equation}
	where $C>0$ is a numerical constant. Then, there exists a time 
	\begin{equation}\label{T0}
	T_0=\frac{1}{2}\min\left(1, \frac{1}{ ( 9\delta_0)^4}, \frac{1}{(3\eta_0)^2}\right)>0,  
	\end{equation}
	and a pair $(\vv, \vartheta)\in \mathcal{C}([0,T_0], \dot{H}^1(\Rt))$ which is the unique solution of the system (\ref{Boussinesq-Evolution}).	
\end{Proposition}	

\begin{proof}
	The system (\ref{Boussinesq-Evolution}) can be rewritten in the following mild formulation: 
	\begin{equation}\label{Boussinesq-Evolution-Mild}
	\begin{cases}\vspace{2mm}
	\vv(t,\cdot)= & \ds{e^{t\Delta} \vv_0 + \int_{0}^{t}e^{(t-\tau)\Delta} \P(\vef)(\tau,\cdot) d \tau-\int_{0}^{t}e^{(t-\tau)\Delta} \P\, \text{div}(\vv \otimes \vv)(\tau,\cdot)  d \tau} \\ \vspace{2mm}
		&+\ds{ \int_{0}^{t}e^{(t-\tau)\Delta} \P (\vartheta\, \veg(\tau,\cdot))d \tau}, \\ 
	\vartheta(t,\cdot)= & \ds{e^{t\Delta}\vartheta_0 + \int_{0}^{t}e^{(t-\tau)\Delta} \eg(\tau,\cdot) d \tau - \int_{0}^{t}e^{(t-\tau)\Delta} \text{div}(\vartheta\, \vv)(\tau,\cdot) d \tau},
	\end{cases}
	\end{equation}	
	where $e^{t \Delta}\varphi := h_t(\cdot)\ast \varphi$, with $h_t(\cdot)$ denoting the well-known heat kernel. 
	
	\medskip
	
	In order to construct a solution $(\vartheta, \vv)$ of this system, we use the following version of the Picard fixed-point scheme. For a proof, see \cite[Theorem $3.2$]{Chamorro-Yangari}.  
	\begin{Theoreme}[Picard's fixed-point]\label{Th:Picard}  	Let \((E,\|\cdot\|_E)\) be a Banach space and $e_0 \in E$. Let $\| e_0 \|_{E}\leq \delta$. Moreover, let $B: E \times E \to E$ be a bilinear form and $L: E \to E$ be a linear form such that,  for any $e,f \in E$, they satisfy. 
		\begin{equation*}
		\| B(e,f) \|_{E} \leq C_B \| e \|_{E}\, \| f \|_{E} \quad \text{and} \| L(e)\|_{E} \leq C_L \| e \|_{E}. 
		\end{equation*}
		If the constants $C_B>0$ and $C_L>0$ satisfy:
		\begin{equation}\label{Conditions-Picard}
		0<C_L<\frac{1}{3}, \quad 0<9\delta C_B <1 \quad \text{and}\quad C_L+6\delta C_B<1,
		\end{equation}	
		then the equation 
		\begin{equation*}
		e=e_0+B(e,e)+L(e),
		\end{equation*}
		admits a solution $e \in E$, which is uniquely determined by $\| e \|_{E}\leq 3\delta$. 
	\end{Theoreme}	

Within the framework of this theorem, we define the expressions:
	\begin{equation}\label{def-e0}
	e:= \left( \begin{array}{c}
	\vv\\ 
	\vartheta
	\end{array}  \right), \qquad  e_0:= \left( \begin{array}{c}\vspace{2mm}
	\ds{e^{t\Delta} \vv_0+ \int_{0}^{t}e^{(t-\tau)\Delta} \P(\vef)(\tau,\cdot) d \tau}\\ 
	\ds{e^{t\Delta}\vartheta_0 + \int_{0}^{t}e^{(t-\tau)\Delta} \eg(\tau,\cdot) d \tau}
	\end{array}  \right),
	\end{equation}
	and 
	\begin{equation}\label{def-B-L}
	B(e,e):= \left( \begin{array}{c}\vspace{3mm}
	\ds{-\int_{0}^{t}e^{(t-\tau)\Delta} \P\, \text{div}(\vv \otimes \vv)(\tau,\cdot)  d \tau}\\ 
	\ds{- \int_{0}^{t}e^{(t-\tau)\Delta} \text{div}(\vartheta\, \vv)(\tau,\cdot) d \tau}
	\end{array}  \right), \quad  L(e):=  \left( \begin{array}{c}\vspace{3mm}
	\ds{ \int_{0}^{t}e^{(t-\tau)\Delta} \P (\vartheta \, \veg(\tau,\cdot))d \tau}\\ 
	0
	\end{array}  \right).
	\end{equation}
	Thereafter, for a time $0<T\leq 1$, which we will later fix as in (\ref{T0}), we consider the Banach space 
	\[ E_T := \left(  \mathcal{C}\big([0,T], \dot{H}^1(\Rt)\big) \times \mathcal{C}\big([0,T], \dot{H}^1(\Rt)\big) ,  \, \| e \|_{E_T}:= \| \vv \|_{L^\infty_t \dot{H}^1_x} +  \| \vartheta \|_{L^\infty_t \dot{H}^1_x} \right). \]
	In the following technical lemmas, we estimate each term defined above in the norm $\|\cdot \|_{E_T}$. 
	
	\begin{Lemme}
		Let $e_0$ be defined as in (\ref{def-e0}) and let $\delta>0$ be the quantity defined in (\ref{delta}). For $0<T\leq 1$ the following estimate holds:
		\begin{equation*}
		\| e_0 \|_{E_T} \leq \delta.
		\end{equation*}
	\end{Lemme}
	The proof of this estimate directly follows from well-known properties of the heat kernel and the first part of Lemma \ref{Lem-Tech:Heat-Estimate} to deal with the external forces terms. 
	
	\begin{Lemme}\label{Lem-Tech-Bilinear-FK}
		Let $B(\cdot , \cdot)$ be the bilinear form defined in (\ref{def-B-L}). Then, it holds that:
		\begin{equation*}
		\| B(e,e)\|_{E_T}\leq  C\, T^{\frac{1}{4}}\| e \|_{E_T}\, \| e \|_{E_T},
		\end{equation*}
		with $C>0$ a numerical constant.
	\end{Lemme}
	
	\begin{proof}
		Note that both components of $B(\cdot,\cdot)$ are completely similar in their structure. Consequently, it is enough to focus on the first one. For $0<t\leq T$ fixed, using well-known properties of the heat kernel and the Leray projector, we write
		\begin{equation}\label{Estim01}
		\begin{split}
		&\, \left\| \int_{0}^{t}e^{(t-\tau)\Delta} \P \text{div}(\vv \otimes \vv)(\tau,\cdot)d\tau \right\|_{\dot{H}^1} \leq \, C\,  \int_{0}^{t} \left\| |\xi| e^{-(t-\tau)|\xi|^2} |\xi|(\widehat{\vv} \ast \widehat{\vv})(\tau,\cdot) \right\|_{L^2}\, d \tau\\
		\leq &\, C\, \int_{0}^{t} \left\| |\xi|^{\frac{3}{2}} e^{-(t-\tau)|\xi|^2}\right\|_{L^\infty}\, \left\| |\xi|^{\frac{1}{2}}  (\widehat{\vv} \ast \widehat{\vv})(\tau,\cdot) \right\|_{L^2}\, d\tau\leq C\, \int_{0}^{t} (t-\tau)^{-\frac{3}{4}}\, \left\| \vv \otimes \vv (\tau,\cdot)\right\|_{\dot{H}^{\frac{1}{2}}}\, d \tau. 
		\end{split}
		\end{equation}	
		Then, by the product laws in homogeneous Sobolev spaces, we obtain
		\begin{equation}\label{Estim02}
		\begin{split}
		&\, C\, \int_{0}^{t} (t-\tau)^{-\frac{3}{4}}\, \left\| \vv \otimes \vv (\tau,\cdot)\right\|_{\dot{H}^{\frac{1}{2}}}\, d \tau \leq C\, \int_{0}^{t} (t-\tau)^{-\frac{3}{4}}\, \| \vv(\tau,\cdot)\|_{\dot{H}^1}\, \| \vv(\tau,\cdot)\|_{\dot{H}^1}\, d\tau \\
		\leq &\, C \left( \int_{0}^{t}(t-\tau)^{-\frac{3}{4}} d \tau\right)\, \left( \sup_{0\leq \tau\leq T} \| \vv(\tau,\cdot)\|_{\dot{H}^1} \right)^2 \leq C\,  T^{\frac{1}{4}} \| e \|^2_{E_T},
		\end{split}
		\end{equation}
		from which the desired estimate follows.  
	\end{proof}
	
	\begin{Lemme}\label{Lem-Tech-Lin-FK}
		Let $L(\cdot)$ be the linear form defined in (\ref{def-B-L}), where $\veg \in \mathcal{C}_t \dot{H}^{\frac{1}{2}}_x$. Then, we have
		\begin{equation*}
		\| L(e)\|_{E_T} \leq C\,  T^{\frac{1}{2}}  \| \veg \|_{L^\infty_t \dot{H}^{\frac{1}{2}}_{x}}\, \| e \|_{E_T}.
		\end{equation*}
	\end{Lemme}
	
	\begin{proof}
		For $0<t \leq T$ fixed, using the second part of Lemma \ref{Lem-Tech:Heat-Estimate}, Hölder inequalities (with $\frac{1}{2}= \frac{1}{6}+\frac{1}{3}$) together with the continuous Sobolev embedding $\dot{H}^{1}(\Rt)\subset L^6(\Rt)$ and $\dot{H}^{\frac{1}{2}}(\Rt)\subset L^3(\Rt)$, we write
		\begin{equation*}
		\begin{split}
		&\, \left\| \int_{0}^{t}e^{(t-\tau)\Delta} \P (\vartheta(\tau,\cdot)\, \veg)d \tau \right\|_{\dot{H}^1} \leq  C\, \| \vartheta\, \veg \|_{L^2_t L^2_x} \leq C\, t^{\frac{1}{2}}\, \left( \sup_{0\leq \tau\leq t}\| \vartheta(\tau,\cdot)\, \veg \|_{L^2}\right) \\
		\leq &\, C\, t^{\frac{1}{2}} \left( \sup_{0\leq \tau \leq t} \| \vartheta(\tau,\cdot)\|_{L^6} \| \veg \|_{ L^3}  \right) \leq C\, T^{\frac{1}{2}}\| \veg \|_{L^\infty_t \dot{H}^{\frac{1}{2}}_{x}}\, \| e \|_{E_T}. 
		\end{split}
		\end{equation*}
	\end{proof}
	
	Returning to the framework of Theorem \ref{Th:Picard}, from Lemmas \ref{Lem-Tech-Bilinear-FK} and \ref{Lem-Tech-Lin-FK} we define $C_B:= C\, T^{\frac{1}{4}}$ and $C_L:= C\,  T^{\frac{1}{2}} \| \veg \|_{L^\infty_t \dot{H}^{\frac{1}{2}}_{x}}$, respectively. Then, setting the time $T$ as in $T_0$, we satisfy all the conditions stated in (\ref{Conditions-Picard}), yielding the existence of a solution $(\vartheta, \vv)\in \mathcal{C}([0,T_0], \dot{H}^1(\Rt))$ to the system (\ref{Boussinesq-Evolution-Mild}).
	
	\medskip
	
	The uniqueness of this solution follows from known arguments as in \cite[Proposition 7.1]{PGLR1}. Proposition \ref{Prop-Tech:Fujita-Kato} is thus proven.
\end{proof}

{\bf Step 2}. \emph{Gevrey estimates for Fujita-Kato mild solutions}.  In the following proposition, we prove that the solution obtained in the proposition above belongs to a certain Gevrey class, provided that the external forces $\vef, \eg$ and the gravitational acceleration $\veg$ satisfy additional Gevrey regularity. To this end, for the parameter $r>0$, recall that we define the operator $e^{r \sqrt{-t \Delta}}(\cdot)$ by the symbol $e^{r \sqrt{t}|\xi|}$. 

\begin{Proposition}\label{Prop-Tech:Fujita-Kato-Gevrey}
	Under the same hypotheses as in Proposition \ref{Prop-Tech:Fujita-Kato}, let $r>0$ be a parameter and assume in addition that 
	\begin{equation}\label{Assumption-Data-Gevrey}
	e^{r \sqrt{-t \Delta}}\left( \vef, \eg\right) \in \mathcal{C}\big( [0,1], \dot{H}^1(\Rt)\big) 
	\quad \text{and}\quad  
	e^{r\sqrt{-t\Delta}}\,\veg \in \mathcal{C}\big([0,1], \dot{H}^{\frac{1}{2}}(\Rt)\big).
	\end{equation}
	Define the quantities 
	\begin{equation}\label{delta1}
	\delta_1:= C ( e^{r^2}+1)\left( \| \vv_0\|_{\dot{H}^1}
	+ \| \vartheta_0\|_{\dot{H}^1} 
	+ \left\| e^{r\sqrt{-t\Delta}}\, \vef \right\|_{L^\infty_t \dot{H}^1_x}
	+ \left\| e^{r\sqrt{-t\Delta}}\, \eg \right\|_{L^\infty_t \dot{H}^1_x}\right), 
	\end{equation}
	and 
	\begin{equation}\label{eta1}
	\eta_1:= C \left\| e^{r\sqrt{-t\Delta}} \, \veg\right\|_{L^\infty_t \dot{H}^{\frac{1}{2}}_x}, 
	\end{equation}
	where $C>0$ is a numerical constant. Then, for the time $0<T_0\leq 1$ given in (\ref{T0}), there exists a time
	\begin{equation}\label{T1}
	T_1= \frac{1}{2} \min\left(1, \frac{1}{(9\delta_1)^4}, \frac{1}{(3\eta_1)^2} \right)<T_0,
	\end{equation}
	such that the solution 
	$(\vv, \vartheta)\in \mathcal{C}([0,T_0], \dot{H}^1(\Rt))$ of the system (\ref{Boussinesq-Evolution}), constructed in Proposition \ref{Prop-Tech:Fujita-Kato}, satisfies
	\begin{equation}\label{Gevrey-Mild-Solutions}
	e^{r \sqrt{-t \Delta}} \big(\vv, \vartheta\big) \in \mathcal{C}\big(]0,T_1], \dot{H}^1 (\Rt)\big).  
	\end{equation} 
\end{Proposition}
\begin{proof}  We begin by explaining the general strategy of the proof. 
	Using the formalism of the fixed-point problem 
	\[ e = e_0 + B(u,u) + L(e), \]
	defined in (\ref{def-e0}) and (\ref{def-B-L}), for a time $0<T \leq T_0 \leq 1$, which we will later fix as $T_1$, we shall construct a solution $e:=(\tilde{v}, \tilde{\vartheta})$ via Picard's fixed-point argument (see again Theorem \ref{Th:Picard}) in the Banach space 
	\[
	F_T :=  \left\{ \varphi \in \mathcal{C}([0,T], \dot{H}^1(\Rt)) : \, e^{r \sqrt{-t \Delta}}\, \varphi \in \mathcal{C}(]0,T], \dot{H}^1(\Rt)) \ \text{and} \ \| \varphi \|_{F_T} := \left\| e^{r \sqrt{-t \Delta}}\, \varphi \right\|_{L^\infty_t \dot{H}^1_x} < +\infty \right\}.
	\]
	Thereafter, since we have the continuous embedding $F_T \subset \mathcal{C}\big([0,T_1], \dot{H}^1(\Rt)\big)$, the uniqueness of solutions in this larger space yields the identity  $(\tilde{v}, \tilde{\vartheta})=(\vv, \vartheta)$ in the interval of time $[0,T_1]$. From this identity, we obtain  the desired conclusion~(\ref{Gevrey-Mild-Solutions}).

\begin{Lemme}\label{Lem-Tech-Data-FK-Gevrey}
	Let $\vv_0, \vartheta_0 \in \dot{H}^1(\Rt)$, $\vef, \eg \in \mathcal{C}\big( [0,1],\dot{H}^{1}(\Rt)\big)$ be the same initial data as in Proposition \ref{Prop-Tech:Fujita-Kato}, and let $e_0$ defined in (\ref{def-e0}).  Assume (\ref{Assumption-Data-Gevrey}) and define the quantity $\delta_1$ as in (\ref{delta1}).   Then, for $0<T\leq 1$ it holds that
	\[
	\| e_0 \|_{F_T} \leq \delta_1.
	\]
\end{Lemme}
\begin{proof}
For $\vv_0 \in \dot{H}^1(\Rt)$, using well-known properties of the heat kernel, for $0<t\leq T$ we write
\begin{equation*}
\begin{split}
\left\| e^{r\sqrt{-t \Delta}} e^{t \Delta} \vv_0 \right\|_{\dot{H}^1} =&\, \left\|\, |\xi| e^{r \sqrt{t}|\xi|}e^{-t|\xi|^2} \widehat{\vv}_0 \right\|_{L^2}\\
=&  \left\|\, |\xi| e^{r |\sqrt{t}\xi|-|\sqrt{t}\xi|^2} \widehat{\vv}_0 \right\|_{L^2\left(|\xi|<\frac{r}{\sqrt{t}}\right)}+\left\|\, |\xi| e^{r |\sqrt{t}\xi|-|\sqrt{t}\xi|^2} \widehat{\vv}_0 \right\|_{L^2\left(|\xi|\geq \frac{r}{\sqrt{t}}\right)}.
\end{split}
\end{equation*}	
For the first term, since $|\xi| \leq \frac{r}{\sqrt{t}}$, we have $e^{r|\sqrt{t}\xi|-|\sqrt{t}\xi|^2}\leq e^{r|\sqrt{t}\xi|}\leq e^{r^2}$. For the second term, since $|\xi| > \frac{r}{\sqrt{t}}$, we obtain $r |\sqrt{t}\xi|-|\sqrt{t}\xi|^2\leq 0$, hence $e^{r |\sqrt{t}\xi|-|\sqrt{t}\xi|^2}  \leq 1$.  

\medskip

The expression $\left\| e^{r\sqrt{-t \Delta}} e^{t \Delta} \vartheta_0 \right\|_{\dot{H}^1}$ follows from the same arguments. Therefore, we can write 
\begin{equation*}
\left\| e^{r\sqrt{-t \Delta}} e^{t \Delta} \vv_0 \right\|_{L^\infty_t \dot{H}^1} + \left\| e^{r\sqrt{-t \Delta}} e^{t \Delta}  \vartheta_0\right\|_{L^\infty_t \dot{H}^1} \leq (e^{r^2}+1)\left( \| \vv_0 \|_{\dot{H^1}}+ \| \vartheta_0 \|_{\dot{H}^1} \right). 
\end{equation*}

On the other hand, by assumption (\ref{Assumption-Data-Gevrey}),  the first part of Lemma \ref{Lem-Tech:Heat-Estimate} and as $0<T\leq 1$, we directly obtain 
\begin{equation*}
\begin{split}
&\, \left\| e^{r\sqrt{-t \Delta}}  \left( \int_{0}^{t}e^{(t-\tau)\Delta} \P(\vef)(\tau,\cdot) d \tau\right) \right\|_{L^\infty_t \dot{H}^1_x} + \left\| e^{r\sqrt{-t \Delta}}  \left( \int_{0}^{t}e^{(t-\tau)\Delta} \eg(\tau,\cdot) d \tau\right) \right\|_{L^\infty_t \dot{H}^1_x} \\
\leq&\, C\left( \left\| e^{r\sqrt{-t\Delta}}\, \vef \right\|_{L^\infty_t \dot{H^1}_x}+ \left\| e^{r\sqrt{-t\Delta}}\, \eg \right\|_{L^\infty_t \dot{H}^1_x} \right).
\end{split}
\end{equation*}
\end{proof}	

\begin{Lemme}\label{Lem-Tech-Bilinear-FK-Gevrey}
	Let $B(\cdot, \cdot)$ be the bilinear form defined in (\ref{def-B-L}). Then it holds that 
	\[
	\| B(e,e) \|_{F_T} \leq C T^{\frac{1}{4}}\| e \|_{F_T}\, \| e\|_{F_T}.
	\]
\end{Lemme}

\begin{proof}
	The proof follows the same arguments as in the proof of Lemma \ref{Lem-Tech-Bilinear-FK}. From the estimates given in (\ref{Estim01}) and well-known properties of the heat kernel, for any $\tilde{v} \in F_T$ we write 
	\begin{equation*}
	\left\| e^{r\sqrt{-t \Delta}}   \left(  \int_{0}^{t}e^{(t-\tau)\Delta} \P \text{div}(\vv \otimes \vv)(\tau,\cdot)d\tau \right) \right\|_{\dot{H}^1} \leq C \int_{0}^{t}(t-\tau)^{-\frac{3}{3}}\, \left\|\, |\xi|^{\frac{1}{2}}\,  e^{r \sqrt{t}|\xi|} (\widehat{\tilde{v}} \ast \widehat{\tilde{v}} )(\tau,\cdot) \right\|_{L^2}\, d\tau. 
	\end{equation*}	
	To control the last expression, observe that for any $\xi, \eta \in \Rt$ we have $e^{r \sqrt{t}|\xi|}\leq e^{r\sqrt{t}|\xi-\eta|} e^{r\sqrt{t}|\eta|}$; hence we obtain the pointwise inequality
	\begin{equation}\label{Pointwise-Fourier-u}
		\left| e^{r \sqrt{t}|\xi|} (\widehat{\tilde{v}} \ast \widehat{\tilde{v}} )(\tau,\xi) \right| \leq \left(  \left( e^{r \sqrt{t}|\xi|}  | \widehat{\tilde{v}}| \right) \ast \left( e^{r \sqrt{t}|\xi|}  | \widehat{\tilde{v}}| \right)  \right)(\tau, \xi). 
	\end{equation}
	Using this inequality and following the same arguments as in (\ref{Estim02}), we have
	\begin{equation*}
	\begin{split}
&\, 	C \int_{0}^{t}(t-\tau)^{-\frac{3}{4}}\, \left\|\, |\xi|^{\frac{1}{2}}\,  e^{r \sqrt{t}|\xi|} (\widehat{\tilde{v}} \ast \widehat{\tilde{v}} )(\tau,\cdot) \right\|_{L^2}\, d\tau \\
\leq &\, C \int_{0}^{t}(t-\tau)^{-\frac{3}{4}}\, \left\|\, |\xi|^{\frac{1}{2}}\,  \left(  \left( e^{r \sqrt{t}|\xi|}  | \widehat{\tilde{v}}| \right) \ast \left( e^{r \sqrt{t}|\xi|}  | \widehat{\tilde{v}}| \right)  \right) (\tau,\cdot) \right\|_{L^2}\, d\tau\\
\leq &\, 	C \int_{0}^{t}(t-\tau)^{-\frac{3}{4}}\, \left\|  \left( e^{r \sqrt{t}|\xi|}  | \widehat{\tilde{v}}| \right)^{\vee}\, \left( e^{r \sqrt{t}|\xi|}  | \widehat{\tilde{v}}| \right)^{\vee} (\tau,\cdot) \right\|_{\dot{H}^{\frac{1}{2}}}\, d\tau \\
\leq &\, C \int_{0}^{t}(t-\tau)^{-\frac{3}{4}}\, \left\|  \left( e^{r \sqrt{t}|\xi|}  | \widehat{\tilde{v}}| \right)^{\vee}(\tau,\cdot) \right\|_{\dot{H}^1}\, \left\|  \left( e^{r \sqrt{t}|\xi|}  | \widehat{\tilde{v}}| \right)^{\vee} (\tau,\cdot) \right\|_{\dot{H}^{1}}\, d\tau\\
=&\,  C \int_{0}^{t}(t-\tau)^{-\frac{3}{4}}\, \left\| e^{t\sqrt{-t \Delta}} \tilde{v} (\tau,\cdot) \right\|_{\dot{H}^1}\, \left\|   e^{t\sqrt{-t \Delta}} \tilde{v} (\tau,\cdot)  \right\|_{\dot{H}^{1}}\, d\tau\\
\leq &\, C\, T^{\frac{1}{4}}\, \| e \|^2_{F_T}.
	\end{split}
	\end{equation*}
	
The other term $\ds{ \int_{0}^{t}e^{(t-\tau)\Delta} \text{div}(\tilde{\vartheta}\, \tilde{v})(\tau,\cdot) d \tau}$ is treated in the same manner.  
\end{proof}

\begin{Lemme}\label{Lem-Tech-Linear-FK-Gegrey} Let $L(\cdot)$ be the linear form defined in (\ref{def-B-L}), where   $\veg \in \mathcal{C}_t \dot{H}^{\frac{1}{2}}_x$.  Assume (\ref{Assumption-Data-Gevrey}) and define $\eta_1$ as in (\ref{eta1}). Then, it holds that 
	\begin{equation*}
	\| L(e)\|_{F_T} \leq \eta_1\, T^{\frac{1}{2}} \,  \| e \|_{F_T}. 
	\end{equation*} 
\end{Lemme}	
\begin{proof}  Applying the second part of Lemma \ref{Lem-Tech:Heat-Estimate} and the  pointwise inequality
	\begin{equation}\label{Pointwise-Fourier-theta}
	\left| e^{r \sqrt{t}|\xi|} (\widehat{\tilde{\vartheta}} \ast \widehat{\veg} )(\tau,\xi) \right| \leq \left(  \left( e^{r \sqrt{t}|\xi|}  | \widehat{\tilde{\vartheta}}| \right) \ast \left( e^{r \sqrt{t}|\xi|}  | \widehat{\veg}| \right)  \right)(\tau, \xi),  
	\end{equation}
we write
\begin{equation*}
\begin{split}
\left\| e^{r \sqrt{-t\Delta}} \left( \int_{0}^{t}e^{(t-\tau)\Delta} \P (\vartheta(\tau,\cdot)\, \veg)d \tau  \right)\right\|_{\dot{H}^1} \leq &\,  C\, \left\| e^{r\sqrt{-\Delta}} \left( \tilde{\vartheta} \, \tilde{v} \right) \right\|_{L^{2}_t L^2_x} \\
\leq &\, C  \left\|   \left( e^{r \sqrt{t}|\xi|}  | \widehat{\tilde{\vartheta}}| \right)^{\vee}\,   \left( e^{r \sqrt{t}|\xi|}  | \widehat{\veg}| \right)^{\vee} \right\|_{L^2_t L^2_x}. 
\end{split}
\end{equation*}		
From  H\"older's inequalities (with $\frac{1}{2}=\frac{1}{6}+\frac{1}{3}$), together with the continuous Sobolev embeddings $\dot{H}^{1}(\Rt)\subset L^6(\Rt)$ and $\dot{H}^{\frac{1}{2}}(\Rt)\subset L^3(\Rt)$, we obtain
\begin{equation*}
\begin{split}
C  \left\|   \left( e^{r \sqrt{t}|\xi|}  | \widehat{\tilde{\vartheta}}| \right)^{\vee}\,   \left( e^{r \sqrt{t}|\xi|}  | \widehat{\veg}| \right)^{\vee} \right\|_{L^2_t L^2_x} \leq &\, C \, t^{\frac{1}{2}} \left\| \left( e^{r \sqrt{t}|\xi|}  | \widehat{\tilde{\vartheta}}| \right)^{\vee}\right\|_{L^\infty_t L^6_x}\,  \left\| \left( e^{r \sqrt{t}|\xi|}  | \widehat{\veg}| \right)^{\vee} \right\|_{L^\infty_t L^3_x} \\
\leq &\, C\, T^{\frac{1}{2}} \left\| \left( e^{r \sqrt{t}|\xi|}  | \widehat{\veg}| \right)^{\vee} \right\|_{L^\infty_t \dot{H}^{\frac{1}{2}}_x}\, \left\| \left( e^{r \sqrt{t}|\xi|}  | \widehat{\tilde{\vartheta}}| \right)^{\vee}\right\|_{L^\infty_t \dot{H}^1_x} \\
=&\,  C\, T^{\frac{1}{2}} \left\| e^{r\sqrt{-t\Delta}} \veg  \right\|_{L^\infty_t \dot{H}^{\frac{1}{2}}_x}\, \left\| e^{r\sqrt{-t\Delta}} \tilde{\vartheta} \right\|_{L^\infty_t \dot{H}^1_x} \\
\leq  &\, \eta_1 \, T^{\frac{1}{2}}\, \| e \|_{F_T}. 
\end{split}
\end{equation*}	
\end{proof}	

With Lemmas \ref{Lem-Tech-Data-FK-Gevrey}, \ref{Lem-Tech-Bilinear-FK-Gevrey}, and \ref{Lem-Tech-Linear-FK-Gegrey} at hand, we set the time $T$ to be $T_1$ given in (\ref{T1}). Proposition \ref{Prop-Tech:Fujita-Kato-Gevrey} now follows from Theorem \ref{Th:Picard}.  
\end{proof}	

{\bf Step 3}. \emph{Gevrey regularity for $\dot{H}^1$-stationary solutions}. 
Let $(\vu, \theta)\in \dot{H}^1(\Rt)$ be a solution of the stationary Boussinesq system (\ref{Boussinesq}) associated with the external forces $\vf, g \in \dot{H}^{-1}(\Rt)$ and the gravitational acceleration $\vvg \in L^{\frac{3}{2}}(\Rt)\cap \dot{H}^{\frac{1}{2}}(\Rt)$.

\begin{Lemme} 
	Assume that $\vf, g$ and $\vvg$ satisfy (\ref{Assumption-Forces-Gevrey}) for a parameter $r>0$. For any $0<t\leq 1$, define
	\[ 	\vef(t,\cdot):= e^{-\frac{2r}{3}\sqrt{-t\Delta}} \left( e^{\frac{2r}{3}\sqrt{-t\Delta}} \vf\right), 
	\quad 
	\eg(t,\cdot):= e^{-\frac{2r}{3}\sqrt{-t\Delta}} \left( e^{\frac{2r}{3}\sqrt{-t\Delta}} g\right), \]
	\[ \veg(t,\cdot):=e^{-\frac{2r}{3}\sqrt{-t\Delta}} \left( e^{\frac{2r}{3}\sqrt{-t\Delta}} \vvg\right). \]
Then it holds that $\vef, \eg \in \mathcal{C}\big([0,1],\dot{H}^{1}(\Rt)\big)$, $\veg \in \mathcal{C}\big([0,1],\dot{H}^{\frac{1}{2}}(\Rt)\big)$, and they satisfy (\ref{Assumption-Data-Gevrey}) with the parameter $\frac{2r}{3}$. More precisely, there exists a constant $C_r>0$, depending only on $r$, such that 
	\begin{equation}\label{Control-Gevrey}
	\left\| e^{\frac{2r}{3}\sqrt{-t \Delta}} \left( \vef, \eg \right)  \right\|_{L^\infty_t \dot{H}^{1}_x} 
	\leq C_r \left\| \left( \vf, g \right)\right\|_{G^{-1}_r} 
	\quad \text{and}\quad 
	\left\| e^{\frac{2r}{3}\sqrt{-t \Delta}} \,  \veg \right\|_{L^\infty_t \dot{H}^{\frac{1}{2}}_x} 
	\leq \| \vvg \|_{G^{\frac{1}{2}}_{r}}. 
	\end{equation}
\end{Lemme}	

\begin{proof} 
	Since $\vef, \eg$ and $\vvg$ are time-independent functions, and since we have the continuous embeddings 
	$G^{1}_{\frac{2r}{3}\sqrt{t}}(\Rt)\subset \dot{H}^{1}(\Rt)$ 
	and 
	$G^{\frac{1}{2}}_{\frac{2r}{3}\sqrt{t}}(\Rt)\subset \dot{H}^{\frac{1}{2}}(\Rt)$, 
	it is sufficient to verify that $\vef, \eg$, and $\veg$ satisfy (\ref{Assumption-Data-Gevrey}) with the parameter $\frac{2r}{3}$. 
	
	\medskip
	
	For $\vef(t,\cdot)$ and $\eg(t,\cdot)$, for any $0<t\leq 1$ we write 
	\begin{equation*}
	\begin{split}
	\left\| \left( \vef(t,\cdot), \eg(t,\cdot) \right) \right\|^2_{G^{1}_{\frac{2r}{3}\sqrt{t}}}
	=&\, \int_{\Rt} |\xi|^2 e^{2 \left(\frac{2r}{3} \sqrt{t}\right) |\xi|}\, 
	\left( | \widehat{\vef}(t,\xi)|^2+| \widehat{\eg}(t,\xi)|^2 \right)d \xi \\
	=&\, \int_{\Rt}  |\xi|^2 e^{2 \left(\frac{2r}{3}\sqrt{t}\right) |\xi|}\, 
	\left( | \widehat{\vf}(\xi)|^2+| \widehat{g}(\xi)|^2 \right)d \xi \\
	\leq &\,  \int_{\Rt}  |\xi|^2 e^{2 \left|\frac{2r}{3}\xi\right|}\, 
	\left( | \widehat{\vf}(\xi)|^2+| \widehat{g}(\xi)|^2 \right)d \xi\\
	=&\,  \int_{\Rt} |\xi|^{-2} |\xi|^4 e^{2 \left|\frac{2r}{3}\xi\right|}\, 
	\left( | \widehat{\vf}(\xi)|^2+| \widehat{g}(\xi)|^2 \right)d \xi\\
	=&\, \left(\frac{3}{2r}\right)^4 \int_{\Rt} |\xi|^{-2} 
	\left|\frac{2r}{3}\xi\right|^4 e^{2 \left|\frac{2r}{3}\xi\right|}\, 
	\left( | \widehat{\vf}(\xi)|^2+| \widehat{g}(\xi)|^2 \right)d \xi\\
	\leq &\, \left(\frac{3}{2r}\right)^4
	\int_{\Rt} |\xi|^{-1}  e^{3 \left|\frac{2r}{3}\xi\right|}\, 
	\left( | \widehat{\vf}(\xi)|^2+| \widehat{g}(\xi)|^2 \right)d \xi\\
	\leq &\, C^2_r\, 
	\left\| \left( \vf, g\right)\right\|^2_{G^{-1}_r}.
	\end{split}
	\end{equation*}	
	
	For  $\veg(t,\cdot)$, as $0<t \leq 1$ we have $\sqrt{t}\leq \frac{3}{2}$. We thus  write
		\begin{equation*}
	\left\| \veg(t,\cdot) \right\|^2_{G^{\frac{1}{2}}_{\frac{2r}{3}\sqrt{t}}}
	=\, \int_{\Rt} |\xi| e^{2 \left(\frac{2r}{3} \sqrt{t}\right) |\xi|}\,\left| \widehat{\veg}(t,\xi)\right|^2 d \xi \leq \,  \int_{\Rt} |\xi|  e^{2 \left|\frac{2r}{3}\xi \frac{3}{2}\right|}\, \left|\widehat{\vvg}(\xi)\right|^2 d \xi  \leq \| \vvg \|^2_{G^{\frac{1}{2}}_r}. 
	\end{equation*}	
	 
\end{proof}

Since we have $\vef, \eg \in \mathcal{C}\big([0,1],\dot{H}^{1}(\Rt)\big)$ and $\veg \in \mathcal{C}\big([0,1],\dot{H}^{\frac{1}{2}}(\Rt)\big)$, within the framework of Proposition \ref{Prop-Tech:Fujita-Kato} we set the initial data $(\vv_0, \vartheta_0)=(\vu, \theta)$, which yields a time $0<T_0\leq 1$ and a unique solution $(\vv,\vartheta)\in \mathcal{C}\big( [0,T_0], \dot{H}^{1}(\Rt)\big)$ for the evolution Boussinesq system (\ref{Boussinesq-Evolution}).

\medskip

Additionally, since $\vef, \eg$, and $\veg$ satisfy (\ref{Assumption-Data-Gevrey}) with the parameter $\frac{2r}{3}$, Proposition \ref{Prop-Tech:Fujita-Kato-Gevrey} ensures the existence of a time $0<T_1<T_0$ such that 
\[
e^{\frac{2r}{3}\sqrt{-t \Delta}} \left( \vv, \vartheta \right) 
\in \mathcal{C}\big( [0,T_1], \dot{H}^1(\Rt)\big).
\]

\medskip

On the other hand, observe that the stationary solution $(\vu, \theta) \in \dot{H}^1(\Rt)$ satisfies $(\vu, \theta) \in \mathcal{C}([0,T_0], \dot{H}^1(\Rt))$ and also solves the evolution Boussinesq system (\ref{Boussinesq-Evolution}) with initial data $(\vu, \theta)$. Therefore, by uniqueness we have $(\vv, \vartheta)=(\vu,\theta)$, and hence we can write 
\[
e^{\frac{2r}{3}\sqrt{-t \Delta}} \left( \vu, \theta \right) 
\in \mathcal{C}\big( [0,T_1], \dot{H}^1(\Rt)\big).
\]

At this point, recall that the time $T_1$, defined in (\ref{T1}), depends on the quantity $\delta_1$ given in (\ref{delta1}), which ultimately depends on the $\dot{H}^1$-norms of the initial data in (\ref{Boussinesq-Evolution}). In the present case, the quantity $\delta_1$ depends on $\| \vu \|_{\dot{H}^1}$ and $\| \theta \|_{\dot{H}^1}$.

\medskip

In order to make the time $T_1$ independent of the solution $(\vu,\theta)$, we proceed as follows. Assuming the energy estimate (\ref{Energy-Estimate}) and using the first control given in (\ref{Control-Gevrey}), we redefine the quantity $\delta_1$ as 
\[
\delta_1:= C ( e^{r^2}+1)\left( 
\|g\|_{\dot H^{-1}}\,\|\vvg\|_{L^{\frac32}}
+\|\vf\|_{\dot {H^{-1}}}
+\| g \|_{\dot{H}^{-1}} 
+ C_r \left\| \left( \vf, g \right)\right\|_{G^{-1}_r} 
\right).
\]
Similarly, for the quantity $\eta_1$ given in (\ref{eta1}), using the second control in (\ref{Control-Gevrey}), we redefine
\[
\eta_1=  C_r \| \vvg \|_{G^{\frac{1}{2}}_{r}}.
\]
Thereafter, the time $T_1$ given in (\ref{T1}) no longer depends on $(\vu,\theta)$.

\medskip

In this setting, since $(\vu, \theta)$ are time-independent functions, by fixing $t=T_1$ and setting $\rho:= \frac{2r}{3}T_1$, we conclude that 
\[
(\vu, \theta) \in G^1_{\rho}(\Rt).
\]

Finally, for the pressure term $P$ characterized in (\ref{Pressure}), from the pointwise inequality (\ref{Pointwise-Fourier-u}) and the product laws in homogeneous Sobolev spaces, we have 
\[
\| P_{\vu} \|_{G^{\frac{1}{2}}_{\rho}} \leq C\| \vu \|^2_{G^{1}_{\rho}}.
\]
Similarly, from the pointwise inequality (\ref{Pointwise-Fourier-theta}) and the fact that $\rho<r$, we obtain
\[
\| P_{\theta}\|_{G^1_\rho} 
\leq C\, \| \theta \|_{G^1_{\rho}}\, \| \vvg \|_{G^{\frac{1}{2}}_{\rho}} 
\leq C\, \| \theta \|_{G^1_{\rho}}\, \| \vvg \|_{G^{\frac{1}{2}}_{r}}.
\]

Theorem \ref{Th:Gevrey} is now proven.

\subsection{Proof of Corollary \ref{Corollary:Holder}}
\subsubsection{First part} For clarity, we divide the proof into the following steps, which are established in the next technical propositions.

\medskip 

{\bf Step 1}. \emph{Global boundedness of $\vu$ and $\theta$}.  
From Theorem \ref{Th:Gevrey} we have that $(\vu, \theta) \in G^1_{\rho}(\Rt)$. Using this fact, we obtain the following result.

\begin{Proposition}\label{Prop-Tech-L-infty} 
	It holds that $\left(\widehat{\vu}, \widehat{\theta}\right) \in L^1(\Rt)$ and consequently $(\vu, \theta)\in L^\infty(\Rt)$. 
\end{Proposition}	
\begin{proof} 
	Since $\widehat{\vu}$ and $\widehat{\theta}$ satisfy the same hypothesis, it is enough to focus on the second function. Applying the Cauchy-Schwarz inequality, we write
	\begin{equation*}
	\int_{\Rt} | \widehat{\theta}(\xi)| \, d \xi 
	= 
	\int_{\Rt} |\xi|^{-1} e^{- \rho |\xi|}\, |\xi| e^{\rho |\xi|}|\widehat{\theta}(\xi)| \, d \xi  
	\leq 
	\left( \int_{\Rt} |\xi|^{-2} e^{-2\rho |\xi|} \, d \xi \right)^{\frac{1}{2}}
	\left(\int_{\Rt} |\xi|^2 e^{2\rho |\xi|} |\widehat{\theta}(\xi)|^2\, d \xi\right)^{\frac{1}{2}}
	<+\infty. 
	\end{equation*}	
\end{proof}

{\bf Step 2}. \emph{Higher-order derivative estimates}. Since $(\vu, \theta)\in \dot{H}^{1}(\Rt)\subset L^6(\Rt)$ and $(\vu, \theta )\in L^\infty(\Rt)$, we have $(\vu,\theta)\in L^p(\Rt)$ for any $6\leq p <+\infty$. Then we will prove that $(\vu,\theta)\in \dot{W}^{k+2,p}(\Rt)$, which yields the desired result by standard interpolation in homogeneous Sobolev spaces.

\begin{Proposition}\label{Prop:W-k+2-p} For $k\in \mathbb{N}$ assume (\ref{Assumption-W-k-infty}). Then, for any $6\leq p <+\infty$, we have $(\vu, \theta)\in \dot{W}^{k+2,p}(\Rt)$. 
\end{Proposition} 
\begin{proof} 
	Observe that $(\vu, \theta)$ can be rewritten as the fixed-point equations:
	\begin{equation}\label{Fixed-Point-System}
	\begin{cases}\vspace{2mm}
	\theta= -(-\Delta)^{-1} \text{div}(\theta\, \vu)+ (-\Delta)^{-1}(g), \\
	\vu= - (-\Delta)^{-1} \mathbb P \text{div}(\vu \otimes \vu)  +(-\Delta)^{-1}(\vf)+(-\Delta)^{-1} \mathbb P\big(\theta \, \vec{{\bf g}}\big).
	\end{cases} 
	\end{equation}
	
	We will use these equations to prove that, for any multi-index $|\alpha| \leq k+2$, we have $(\partial^\alpha\vu,  \partial^\alpha\theta)\in L^p(\Rt)$. To this end, in the following technical lemmas we perform an iteration process with respect to $|\alpha|$. 

\begin{Lemme}[The initial case for $k=0$]\label{Lem-Tech-Reg-Initial} 
Assuming (\ref{Assumption-W-k-infty}) and since $(\vu, \theta)\in \dot{H}^1(\Rt)\cap L^\infty(\Rt)$, for any $|\alpha|\leq 2$ and $6\leq p <+\infty$, we have $(\partial^\alpha\vu, \partial^\alpha \theta) \in L^p(\Rt)$.
\end{Lemme}	 
\begin{proof}  For clarity, we consider the cases $|\alpha|=1$ and $|\alpha|=2$ separately. 
	\begin{itemize}
		\item Let $|\alpha|=1$. From the first equation in (\ref{Fixed-Point-System}), we write
		\begin{equation*}
		\partial^\alpha \theta= -(-\Delta)^{-1} \partial^\alpha \text{div}(\theta\, \vu)+ (-\Delta)^{-1} \partial^\alpha(g), 
		\end{equation*}
		where we must verify that each term on the right-hand side belongs to $L^p(\Rt)$. 
		
		\medskip
		
	For the first term, since $(\vu, \theta )\in L^p(\Rt) \cap L^\infty(\Rt)$ we have $\theta \vu \in L^p(\Rt)$, and since the operator $ -(-\Delta)^{-1} \partial^\alpha \text{div}(\cdot)$ can be written as a linear combination of the Riesz transforms $\mathcal{R}_i \mathcal{R}_j$ (where $\mathcal{R}_i:= (-\Delta)^{-\frac{1}{2}}\partial_i$), we obtain that $-(-\Delta)^{-1} \partial^\alpha \text{div}(\theta\, \vu) \in L^p(\Rt)$.
		
		\medskip
		
For the second term, since $g\in G^{-1}_r(\Rt)$ also satisfies (\ref{Assumption-W-k-infty}), we have 
\[
g\in G^{-1}_{r}(\Rt)\cap \dot{W}^{-1,\infty}(\Rt)\subset \dot{H}^{-1}(\Rt)\cap \dot{W}^{-1,\infty}(\Rt)\subset \dot{W}^{-1,p}(\Rt),
\]
which implies that $(-\Delta)^{-1} \partial^\alpha(g) \in L^p(\Rt)$. Collecting these facts, we obtain that $\theta \in \dot{W}^{1,p}(\Rt)$.

\medskip
		
	Similarly, from the second equation in (\ref{Fixed-Point-System}) we write
	\begin{equation*}
	\partial^\alpha \vu= - (-\Delta)^{-1} \mathbb P 	\partial^\alpha \text{div}(\vu \otimes \vu)  +(-\Delta)^{-1}	\partial^\alpha(\vf)+(-\Delta)^{-1} \mathbb{P}  	\partial^\alpha\big(\theta \, \vec{{\bf g}}\big).
	\end{equation*}
	The first and the second terms follow from the same arguments as above. For the third term, observe first that we have $\vvg \in G^{\frac{1}{2}}_r (\Rt)\subset \dot{H}^{\frac{1}{2}}(\Rt)\subset L^3(\Rt)$. Then, by well-known properties of the Riesz transform, the Hardy--Littlewood--Sobolev inequalities and H\"older's inequality, and since $\theta \in L^p(\Rt)$ for any $6\leq p <+\infty$, we obtain
	\begin{equation*}
	\| (-\Delta)^{-1}\P \partial^\alpha (\theta\, \vvg)\|_{L^p} \leq C \| (-\Delta)^{-\frac{1}{2}} (\theta\, \vvg)\|_{L^p} \leq C \| \theta\, \vvg \|_{L^{\frac{3p}{3+p}}} \leq C \| \theta \|_{L^p}\, \| \vvg \|_{L^3}. 
	\end{equation*}
	We thus obtain that $\vu \in \dot{W}^{1,p}(\Rt)$.
	
	\medskip
		
	\item Let $|\alpha|=2$. We split $\alpha=\alpha_1 + \alpha_2$, where $|\alpha_1|=|\alpha_2|=1$. Then, since $(\vu, \theta)\in \dot{W}^{1,p}(\Rt)\cap L^\infty(\Rt)$, we can write
	\begin{equation*} 
	\begin{split}
	\| (-\Delta)^{-1} \partial^\alpha \text{div}(\theta\, \vu)\|_{L^p}=&\,  	\| (-\Delta)^{-1} \partial^{\alpha_1} \text{div}\,  \partial^{\alpha_2}(\theta\, \vu)\|_{L^p} \leq C\, \| \partial^{\alpha_2}(\theta\, \vu) \|_{L^p}\\
	\leq&\,  C\,\left( \| \vu \|_{\dot{W}^{1,p}}\| \theta \|_{L^\infty}+ \| \vu\|_{L^\infty}\, \| \theta \|_{\dot{W}^{1,p}}\right).
	\end{split} 
	\end{equation*}
	Similarly, we have 
	\begin{equation*}
	\| (-\Delta)^{-1} \P \partial^\alpha \text{div}(\vu\otimes \vu)\|_{L^p}\leq C \| \vu \|_{\dot{W}^{1,p}}\, \| \vu \|_{L^\infty}. 
	\end{equation*}
	
	For the remaining terms in the system (\ref{Fixed-Point-System}), observe that the operators $ (-\Delta)^{-1}\P \partial^\alpha (\cdot)$ and $ (-\Delta)^{-1} \partial^\alpha (\cdot)$ can be written as linear combinations of Riesz transforms. Consequently, the desired result follows from the already known properties of $\theta, \vf, g$ and $\vvg$.
	\end{itemize}	
\end{proof}

\begin{Lemme}[The iterative process for $k\geq 1$]\label{Reg-Tech-Reg-Iteration}  
	Assume (\ref{Assumption-W-k-infty}). In addition, for $1\leq m \leq k$ and for any $|\alpha|\leq m$, assume that $(\partial^\alpha \vu, \partial^\alpha \theta)\in L^p(\Rt)$ for any $6\leq p <+\infty$. Then the result also holds for $|\alpha|\leq k+2$. 
\end{Lemme}	
\begin{proof} The proof follows very similar arguments to those used in the proof of the previous lemma. For the reader's convenience, we only sketch the main estimates for $|\alpha|\leq k+1$. 

\medskip

	In the system (\ref{Fixed-Point-System}), we first consider the bilinear terms involving $\vu$ and $\theta$, then the linear term involving $\theta$ and $\vvg$, and finally the data terms involving $\vf$ and $g$.
	
\medskip	 
	
By splitting $\alpha=\alpha_1+\alpha_2$, where $|\alpha_1|=1$ and $|\alpha_2|=k$, and using the Leibniz rule, we write
\begin{equation*}
\begin{split}
\left\| (-\Delta)^{-1} \partial^\alpha \, \text{div}(\theta\, \vu) \right\|_{L^p} = &\,  \left\|  (-\Delta)^{-1} \partial^{\alpha_1} \, \text{div} \, \partial^{\alpha_2}(\theta\, \vu)\right\|_{L^p}  \leq C\,\| \partial^{\alpha_2}(\theta\, \vu)\|_{L^p}\\
\leq &\, C\, \sum_{|\beta|\leq k} C_{\alpha_2,\beta} \left\| \partial^\beta \theta\, \partial^{\alpha_2-\beta}\vu \right\|_{L^p},
\end{split}
\end{equation*}
where $C_{\alpha_2,\beta}>0$ is a numerical constant depending on the multi-indices $\alpha_2$ and $\beta$. From our assumption on $\vu$ and $\theta$, we have $\partial^\beta \theta \in L^{2p}(\Rt)$ and $\partial^{\alpha_2-\beta}\vu \in L^{2p}(\Rt)$ (just write $2p$ instead of $p$). Therefore, by H\"older's inequality, we obtain 
\begin{equation*}
C\, \sum_{|\beta|\leq k} C_{\alpha_2,\beta} \left\| \partial^\beta \theta\, \partial^{\alpha_2-\beta}\vu \right\|_{L^p} \leq C\, \sum_{|\beta|\leq k} C_{\alpha_2,\beta} \left\| \partial^\beta \theta \|_{L^{2p}}\, \| \partial^{\alpha_2-\beta}\vu \right\|_{L^{2p}}<+\infty.
\end{equation*}
As before, the term $(-\Delta)^{-1} \mathbb P \text{div}(\vu \otimes \vu)$ follows the same estimates. 

\medskip 

For the term $(-\Delta)^{-1} \mathbb{P}  	\partial^\alpha\big(\theta \, \vec{{\bf g}}\big)$, we split $\alpha=\alpha_1+\alpha_2$, where $|\alpha_1|=2$ and $|\alpha_2|=k-1$. Then, using the same arguments as above we write
\begin{equation*}
\left\| (-\Delta)^{-1} \mathbb{P}  	\partial^\alpha\big(\theta \, \vec{{\bf g}}\big)\right\|_{L^P} \leq C\, \| \partial^{\alpha_2}(\theta\, \vvg)\|_{L^p} \leq C\, \sum_{|\beta|\leq k-1} C_{\alpha_2,\beta}\| \partial^\beta \theta\|_{L^{2p}}\, \| \partial^{\alpha_2-\beta} \vvg \|_{L^{2p}}.
\end{equation*}
Here, the term $ \| \partial^{\alpha_2-\beta} \vvg \|_{L^{2p}}$ is well-controlled since we can  verify that 
\begin{equation}\label{Reg-Gravitational-Acceletation}
\vvg \in \dot{W}^{m,p}(\Rt), \qquad \text{for any $1\leq m\leq k$ and $6\leq p <+\infty$}.
\end{equation}
Indeed, since $\vvg \in G^{\frac{1}{2}}_r(\Rt)$, we first obtain that $\vvg \in \dot{W}^{m,2}(\Rt)$ by writing
\begin{equation*}
\begin{split}
\int_{\Rt} |\xi|^{2m} |\widehat{\vvg}(\xi)|^2 d \xi = &\,  \int_{\Rt} |\xi|^{2m-1} e^{-2 r |\xi|}\, |\xi|e^{2r|\xi|} |\widehat{\vvg}(\xi)|^2 d \xi \\
\leq &\, \left(  \sup_{\xi \in \Rt} |\xi|^{2m-1} e^{-2 r |\xi|}\right) \, \int_{\Rt}|\xi|e^{2r|\xi|} |\widehat{\vvg}(\xi)|^2 d\xi <+\infty.
\end{split} 
\end{equation*}
Thereafter, from assumption (\ref{Assumption-W-k-infty}) we also have that $\vvg \in \dot{W}^{m,\infty}(\Rt)$, which yields $\vvg \in \dot{W}^{m,p}(\Rt)$ for any $6\leq p <+\infty$. 

\medskip

For the term $(-\Delta)^{-1}\partial^{\alpha}(\vf, g)$, recall that we have $(\vf, g)\in G^{-1}_r(\Rt)$ and, by assumption (\ref{Assumption-W-k-infty}), we also have $(\vf, g)\in \dot{W}^{k-1,\infty}(\Rt)$. Therefore, following the arguments above, we obtain $(\vf, g)\in \dot{W}^{k-1,p}(\Rt)$. Thereafter, we split $\alpha= \alpha_1+\alpha_2$, where $|\alpha_1|=2$ and $|\alpha_2|=k-1$, to get
\begin{equation*}
\| (-\Delta)^{-1}\partial^{\alpha}(\vf, g) \|_{L^p} \leq C\,  \|  \partial^{\alpha_2}(\vf, g) \|_{L^p} \leq C\| (\vf, g) \|_{\dot{W}^{k-1,p}}. 
\end{equation*}

We thus obtain that $(\vu, \theta)\in \dot{W}^{k+1,p}(\Rt)$ and, by repeating this process, we reach the desired conclusion $(\vu, \theta)\in \dot{W}^{k+2,p}(\Rt)$.
\end{proof}	
From Lemmas \ref{Lem-Tech-Reg-Initial} and \ref{Reg-Tech-Reg-Iteration}, we conclude the proof of Proposition \ref{Prop:W-k+2-p}.  \end{proof}

To complete the proof of the first part of Corollary \ref{Corollary:Holder}, we study the regularity of the pressure $P$.		

\begin{Lemme} Let $P$ be the pressure term characterized in (\ref{Pressure}) through $\vu, \theta$ and $\vvg$. Since $(\vu, \theta)\in L^p(\Rt) \cap L^\infty(\Rt) \cap \dot{W}^{k+2,p}(\Rt)$ and $\vvg$ satisfies (\ref{Assumption-W-k-infty}), it holds that $P\in L^p(\Rt)\cap \dot{W}^{k+2,p}(\Rt)+L^p(\Rt)\cap \dot{W}^{k-1,p}(\Rt)$.
\end{Lemme}	

\begin{proof} 
From (\ref{Pressure}) we have
\[
P=  (-\Delta)^{-1}\text{div} \left( \text{div}(\vu \otimes \vu)\right) - (-\Delta)^{-1} \text{div}(\theta \, \vvg):= P_{\vu}+P_{\theta}.
\]

For the first term on the right-hand side, since $\vu \in L^p(\Rt) \cap L^\infty(\Rt) \cap \dot{W}^{k+2,p}(\Rt)$, applying H\"older inequalities and Lemma \ref{Leibniz-rule} (with $p_1=q_2=p$, $p_2=q_1=\infty$ and $s=k+2$) we obtain that $P_{\vu}\in L^p(\Rt)\cap  \dot{W}^{k+2,p}(\Rt)$. 

\medskip

For the second term, on the one hand recall that $\theta \in L^p(\Rt)\cap L^\infty(\Rt)\cap \dot{W}^{k+2,p}(\Rt)$. On the other hand, since $\vvg \in G^{\frac{1}{2}}_r(\Rt)$, following the same arguments as in the proof of Proposition \ref{Prop-Tech-L-infty}, we have $\vvg \in L^\infty(\Rt)$. Additionally, by (\ref{Reg-Gravitational-Acceletation}) we have $\vvg \in \dot{W}^{k,p}(\Rt)$. Thereafter, using again H\"older inequalities and Lemma \ref{Leibniz-rule}, we obtain $\theta \, \vvg \in L^p(\Rt)\cap \dot{W}^{k,p}(\Rt)$, which yields $P_\theta \in \dot{W}^{1,p}(\Rt) \cap \dot{W}^{k+1,p}(\Rt)$.
\end{proof}

The first part of Corollary \ref{Corollary:Holder} is thus proven.

\subsubsection{Second part} H\"older regularity of weak $\dot{H}^1$-solutions is now a direct consequence of the first part of Corollary \ref{Corollary:Holder} and Lemma \ref{Lem-Holder-Reg}. 

\medskip

Recall that the Lebesgue space $L^p(\Rt)$ continuously embeds into the Morrey space $\dot{M}^{1,p}(\Rt)$ defined in (\ref{Morrey}). Therefore, for any $0\leq s \leq k+1$ and $6\leq p <+\infty$, defining $\sigma:= 1-\frac{3}{p}$, we write
\begin{equation*}
\left| (-\Delta)^{\frac{s}{2}}\big( \theta(x)-\theta(y)\big)\right| \leq \,  C\, \| \vec{\nabla} (-\Delta)^{\frac{s}{2}} \theta \|_{\dot{M}^{1,p}}\, |x-y|^{1-\frac{3}{p}} \leq C\, \left\| (-\Delta)^{\frac{s+1}{2}}\theta \right\|_{L^p}\, |x-y|^{\sigma}.
\end{equation*}

Similarly, for the velocity $\vu$ and the pressure $P$ we obtain that $\vu\in \mathcal{C}^{s,\sigma}(\Rt)$ and $P\in \mathcal{C}^{s,\sigma}(\Rt)+\mathcal{C}^{\min(s,k),\sigma}(\Rt)$, respectively. The second part of Corollary \ref{Corollary:Holder} is now proven.

\section{Analyticity of weak $\dot{H}^1$-solutions in the homogeneous case: proof of Theorem \ref{Th:Gevrey-Homogeneous}}\label{Sec:Homogeneous} 

The proof essentially follows the same steps as in the proof of Theorem \ref{Th:Gevrey}. Consequently, we only detail the main points.

\medskip

First, we state a unified analogous version of Propositions \ref{Prop-Tech:Fujita-Kato} and \ref{Prop-Tech:Fujita-Kato-Gevrey} for the evolution homogeneous Boussinesq system:
\begin{equation}\label{Boussinesq-Evolution-Homogeneous} 
\begin{cases}\vspace{2mm}
\partial_t \vv - \Delta \vv + \P\, \text{div}(\vv \otimes \vv) = \P (\vartheta\, \veg),  \quad \text{div}(\vv)=0, \\ \vspace{2mm}
\partial_t \vartheta - \Delta\vartheta + \text{div}(\vartheta\, \vv)= 0, \\
\vv(0,\cdot)=\vv_0, \quad \vartheta(0,\cdot)=\vartheta_0. 
\end{cases}
\end{equation}

\begin{Proposition}   
	Let $\vv_0, \vartheta_0 \in \dot{H}^1(\Rt)$ with $\text{div}(\vv)=0$, and let $\veg \in \mathcal{C}\big([0,1], \dot{H}^{\frac{1}{2}}(\Rt)\big)$. Then the following statements hold:
	\begin{enumerate}
		\item Define the quantities
		\[
		\delta_0:= C\left( \| \vv_0\|_{\dot{H}^1}+\| \vartheta_0 \|_{\dot{H}^1} \right), 
		\quad 
		\eta_0:= C \| \veg \|_{L^\infty_t \dot{H}^{\frac{1}{2}}_x}, 
		\quad 
		\text{and}
		\quad 
		T_0:= \frac{1}{2}\min\left( 1, \frac{1}{(9\delta_0)^4}, \frac{1}{(3\eta_0)^2} \right).
		\]
		Then the system (\ref{Boussinesq-Evolution-Homogeneous}) has a unique solution 
		$(\vv, \vartheta)\in \mathcal{C}\big( [0,T_0], \dot{H}^1(\Rt)\big)$.
		
		\medskip
		
		\item Let $r>0$, and assume in addition that $\veg \in G^{\frac{1}{2}}_r(\Rt)$. Define
		\[
		\delta_1:= C(e^{r^2}+1)\left( \| \vv_0\|_{\dot{H}^1}+\| \vartheta_0 \|_{\dot{H}^1} \right), 
		\quad 
		\eta_1:= C \| e^{r\sqrt{-t\Delta}}\veg \|_{L^\infty_t \dot{H}^{\frac{1}{2}}_x},
		\]
		\[
		T_1:= \frac{1}{2}\min\left( T_0, \frac{1}{(9\delta_1)^4}, \frac{1}{(3\eta_1)^2} \right).
		\]
		Then the solution obtained in the first part satisfies
		\[
		e^{r \sqrt{-t \Delta}} \big(\vv, \vartheta\big) 
		\in \mathcal{C}\big(]0,T_1], \dot{H}^1 (\Rt)\big).
		\]
	\end{enumerate}
\end{Proposition}

Within this framework, for any solution $(\vu, \theta)\in \dot{H}^1(\Rt)$ of the homogeneous Boussinesq system (\ref{Boussinesq-Homogeneous}), we proceed as in Step~3 of the proof of Theorem \ref{Th:Gevrey}. Nevertheless, in this case the time $T_1$ depends on $\| \vu \|_{\dot{H}^1}$, $\| \theta \|_{\dot{H}^1}$, and $\| \vvg \|_{G^{-\frac{1}{2}}_r}$. Consequently, the quantity 
\[
\varrho:= \frac{2r}{3} T_1,
\]
also depends on $\| \vu \|_{\dot{H}^1}$ and $\| \theta \|_{\dot{H}^1}$. This completes the proof of Theorem \ref{Th:Gevrey-Homogeneous}. 

\section{Liouville-type theorem for weak $\dot{H}^1$-solutions: proof of Theorem \ref{Th:Liouville}}\label{Sec:Liouville} 

Let $(\vu, \theta)\in \dot{H}^1(\Rt)$ be a weak solution of the system (\ref{Boussinesq-Evolution-Mild}), with $\vvg \in G^{\frac{1}{2}}_r(\Rt)$ for some fixed $r>0$. Then, by Theorem \ref{Th:Gevrey-Homogeneous}, there exists $0<\rho<\frac{2r}{2}$ such that 
\[ (\vu, \theta) \in G^1_{\rho}(\Rt). \]
With this fact, from Proposition \ref{Prop-Tech-L-infty} it holds that 
\[ \left( \widehat{\vu}, \widehat{\theta}  \right) \in L^1(\Rt).  \]

Using this information, we can prove the following technical result. Here we recall that the homogeneous Besov space $\dot{B}^{-1}_{\infty,\infty}(\Rt)$ is defined as the space of tempered distributions $\varphi \in \mathcal{S}'(\Rt)$ satisfying 
\[ \| \varphi \|_{\dot{B}^{-1}_{\infty,\infty}}:= \sup_{t>0} t^{\frac{1}{2}}\, \| e^{t\Delta} \varphi \|_{L^\infty}<+\infty. \]

\begin{Lemme} Let $\left( \widehat{\vu}, \widehat{\theta}  \right) \in L^1(\Rt)$. Assume in addition that (\ref{Assumption-Liouville}) holds. Then, we have $\left( \vu, \theta\right) \in \dot{B}^{-1}_{\infty,\infty}(\Rt)$.  
\end{Lemme}	 

\begin{proof} Let $t>0$. Then, we write
	\begin{equation*}
	\begin{split}
	t^{\frac{1}{2}}\, \left\| e^{t\Delta} \left( \vu, \theta \right)   \right\|_{L^\infty} \leq &\,  t^{\frac{1}{2}}\, \left\| e^{-t|\xi|^2} \left( \widehat{\vu},  \widehat{\theta}\right)   \right\|_{L^1} =   t^{\frac{1}{2}}\, \left\||\xi| e^{-t|\xi|^2}\, |\xi|^{-1}  \left( \widehat{\vu},  \widehat{\theta}\right)   \right\|_{L^1} \\
	=&\, \left\|\left| t^{\frac{1}{2}}\, \xi\right| e^{- \left|t^{\frac{1}{2}}\xi\right|^2}\, |\xi|^{-1}  \left( \widehat{\vu},  \widehat{\theta}\right) \right\|_{L^1}  \leq \left\|  \left| t^{\frac{1}{2}}\, \xi\right| e^{- \left|t^{\frac{1}{2}}\xi\right|^2} \right\|_{L^\infty}\, \left\|\, |\xi|^{-1} \left( \widehat{\vu}, \widehat{\theta}\right) \right\|_{L^1} \\
	\leq &\, C\, \left\|\, |\xi|^{-1} \left( \widehat{\vu}, \widehat{\theta}\right) \right\|_{L^1}.
	\end{split}
	\end{equation*}	
	
	To control this last expression, we split
	\begin{equation*}
	\left\|\, |\xi|^{-1} \left( \widehat{\vu}, \widehat{\theta}\right) \right\|_{L^1} = \left\|\, |\xi|^{-1} \left( \widehat{\vu}, \widehat{\theta}\right) \right\|_{L^1(|\xi|\leq 1)}+ \left\|\, |\xi|^{-1} \left( \widehat{\vu}, \widehat{\theta}\right) \right\|_{L^1(|\xi|>1)}=:I_1+ I_2,
	\end{equation*}
	where we estimate each term separately. 
	
	\medskip
	
	To study $I_1$, for any $k\in \mathbb{N}$ we consider the annulus $\mathscr{C}_k$ defined in (\ref{C_k}) and use the following dyadic decomposition
	\[ I_1= \sum_{k=0}^{\infty} \int_{\mathscr{C}_k} |\xi|^{-1}\, \left( \widehat{\vu}, \widehat{\theta}\right)(\xi) d \xi.  \]
	
	By definition of $\mathscr{C}_k$, for any $\xi \in \mathscr{C}_k$ we have $2^{-(k+1)}\leq |\xi|\leq 2^{-k}$, hence $|\xi|^{-1}\leq 2^{k+1}$. Then we obtain 
	\begin{equation*}
	\begin{split}
	I_1 \leq &\, \sum_{k=0}^{\infty} 2^{k+1} \int_{\mathscr{C}_k} \left( \widehat{\vu}, \widehat{\theta}\right)(\xi) d \xi \leq C\, \sum_{k=0}^{+\infty} 2^{k+1} \, \left\| \left( \widehat{\vu}, \widehat{\theta}\right) \right\|_{L^\infty(\mathscr{C}_k)}\,\left( \int_{\mathscr{C}_k} d \xi\right) \\
	\leq &\,  C\, \sum_{k=0}^{+\infty} 2^{k+1} \, \left\| \left( \widehat{\vu}, \widehat{\theta}\right) \right\|_{L^\infty(\mathscr{C}_k)}\, 2^{-3k}
	\leq \, C\, \sum_{k=0}^{+\infty} 2^{-2k+1} \, \left\| \left( \widehat{\vu}, \widehat{\theta}\right) \right\|_{L^\infty(\mathscr{C}_k)}.
	\end{split}
	\end{equation*}
	
	Finally, by assumption (\ref{Assumption-Liouville}), we have
	\[ I_1 \leq  C\, \sum_{k=0}^{+\infty} 2^{-2k+1} \, 2^k = C\, \sum_{k=0}^{+\infty} 2^{-k+1} <+\infty. \]
	
	To study $I_2$, since $\left( \widehat{\vu}, \widehat{\theta}  \right) \in L^1(\Rt)$ and $|\xi|>1$, we directly have 
	\[ I_2 \leq \left\| \left( \widehat{\vu}, \widehat{\theta}\right) \right\|_{L^1(|\xi|>1)} <+\infty.\]
\end{proof}	

\medskip

Once we have that $\left( \vu, \theta\right) \in \dot{B}^{-1}_{\infty,\infty}(\Rt)$, the proof of Proposition \ref{Prop:Liouville} concludes as follows. As we also have $\left( \vu, \theta\right) \in \dot{H}^1(\Rt)$, applying the improved Sobolev inequalities (see \cite{Gerard}) we can write 
\[ \left\| \left( \vu, \theta\right) \right\|_{L^4} \leq C \left\| \left( \vu, \theta\right) \right\|^{\frac{1}{2}}_{\dot{B}^{-1}_{\infty,\infty}}\, \left\| \left( \vu, \theta\right) \right\|^{\frac{1}{2}}_{\dot{H}^1} <+\infty. \] 

Returning to the second equation in the Boussinesq system (\ref{Boussinesq-Homogeneous}), as $\left( \vu, \theta\right) \in L^4(\Rt)$, a simple computation yields that $\text{div}(\theta\, \vu)\in \dot{H}^{-1}(\Rt)$. From a standard energy estimate we find that $\int_{\Rt} | \vec{\nabla}\theta |^2\, dx =0$, which implies that $\theta =0$ due to the well-known Sobolev embedding $\dot{H}^1(\Rt) \subset L^6(\Rt)$. 

\medskip

Once we have $\theta=0$, the first equation of the Boussinesq system (\ref{Boussinesq-Homogeneous}) reduces to the classical Navier–Stokes equations, where the fact that $\vu \in L^4(\Rt)$ together with its divergence-free property yields that $\vu=0$. See  \cite[Theorem $1$]{Chamorro-Jarrin-Lemarie}. 

\medskip

Theorem \ref{Th:Liouville} is proven.

\medskip 

\paragraph{{\bf Statements and Declaration}}
Data sharing does not apply to this article as no datasets were generated or analyzed during the current study.  In addition, the authors declare that they have no conflicts of interest, and all of them have equally contributed to this paper.

%%%%%%%%%%%%%%%%%%%%%%%%%%%%%%%%%%%%%%%%%%%%%%

\vspace{1cm} 


\begin{thebibliography}{40}
	\bibitem{Avecedo} P. Acevedo, C. Amrouche \& C. Conca. \emph{Boussinesq system with non-homogeneous boundary conditions}. Applied Mathematics Letters, Volume 53: 39-44 (2016). 
	%%%%%%%%
	\bibitem{Biswas} A.  Biswas \emph{ Gevrey regularity for a class of dissipative equations with applications to
	decay}. J. Differential Equations, 253: 2739-2764 (2012).
 	%%%%%%%%
 	\bibitem{Bourgain} J. Bourgain and N. Pavlovi\'c. \emph{ Ill-posedness of the Navier–Stokes equations in a critical space in 3D}, J. Funct. Anal. 255 no. 9, 2233-2247. (2008).
 	%%%%%%%%
	\bibitem{Boussinesq} J. Boussinesq. \emph{Théorie analytique de la chaleur: mise en harmonie avec la thermodynamique et avec la théorie mécanique de la lumière}. Volume 2. Gauthier-Villars, Paris, (1903).
	%%%%%%%
	\bibitem{Brezis}  H. Brezis. \emph{Analyse Fonctionnelle, Théorie et applications}. Dunod, (1999).
	%%%%%%%	
	\bibitem{Chae} D. Chae and J. Wolf. \emph{On Liouville type theorems for the steady Navier–Stokes equations in  $\Rt$}.  J. Differ. Equ. 261, 5541–5560, (2016).
	%%%%%%%
	\bibitem{Chamorro-Jarrin-Lemarie-Gevrey} D. Chamorro, O. Jarrín, and P.-G. Lemarié-Rieusset. \emph{Frequency decay for Navier–Stokes stationary solutions} C. R. Math. Acad. Sci. Paris 357 (2), 175–179, (2019).
	%%%%%%%
	\bibitem{Chamorro-Jarrin-Lemarie} D. Chamorro, O. Jarrín, and P.-G. Lemari\'e-Rieusset. \emph{Some Liouville theorems for stationary Navier–Stokes equations in Lebesgue and Morrey spaces}. Ann. Inst. H. Poincaré Anal. Non Linéaire 38 (3), 689–710 (2021).
	%%%%%%%
	\bibitem{Chamorro-Yangari} D. Chamorro and M. Yangari. \emph{ Some existence and regularity results for a non-local transport-diffusion equation with fractional derivatives in time and space}. J. Differ. Equ. 428: 389–421  (2025). 
	%%%%%%%
	\bibitem{Chandrasekhar} S. Chandrasekhar. \emph{Hydrodynamic and Hydromagnetic Stability}. The International Series of Monographs on Physics. Clarendon Press, Oxford, (1961).
	%%%%%%%
	\bibitem{Cortez-Jarrin}  M. F. Cortez and O. Jarrín and. \emph{ On the long-time behavior for a damped Navier-Stokes-Bardina model}.  Discrete Contin. Dyn. Syst. 42, no. 8, 3661–3707 (2022).
	%%%%%%%
	\bibitem{Foias-Temam} C.  Foias and R. Temam. \emph{Gevrey class regularity for the solutions of the Navier-Stokes
	equations}. J. Funct. Anal, 87: 359-369 (1989).
	%%%%%%%
	\bibitem{Gerard} P. G\'erard, Y. Meyer and F. Oru. \emph{Improved Sobolev inequalities}. Seminary of partial differential equations.
	Vol. 1-8 (1996-1997). 
	%%%%%%%
	\bibitem{GigaMiyakawa} Y. Giga and T. Miyakawa. \emph{Navier-stokes flow in $\Rt$ with measures as initial vorticity and morrey spaces}. Communications in Partial Differential Equations, 14:5, 577-618 (1989).
	%%%%%%%
		\bibitem{Gil}  M.I. Gil. \emph{Solvability of a system of stationary Boussinesq equations}.  J. Differential Equations 271: 1377–1382  (1991).
	%%%%%%%
	\bibitem{Grafakos} L. Grafakos and S. Oh. \emph{The Kato–Ponce inequality}. Comm. Partial Differential Equations 39, no. 6, 1128-1157 (2014).
	%%%%%%%
	\bibitem{Guo} B. Guo and B. Wang. \emph{Gevrey class regularity and approximate inertial manifolds for the Newton–Boussinesq equations} Chinese Ann. Math. Ser. B 19 (2), 179–188, (1998).
	%%%%%%%
	\bibitem{Jarrin1} O. Jarrín. \emph{A remark on the Liouville problem for stationary Navier–Stokes equations in Lorentz and Morrey spaces}. J. Math. Anal. Appl. 486, 123871, (2020)
	%%%%%%%
	\bibitem{Jarrin2} O. Jarrín. \emph{A short note on the Liouville problem for the steady-state Navier–Stokes equations}. Arch. Math. 121, 303–315, (2023). 
	%%%%%%%
	\bibitem{Kalantarov} V. K. Kalantarov, B. Levant, and E. S. Titi. \emph{ Gevrey regularity for the attractor of the 3D Navier–Stokes–Voigt equations}. J. Nonlinear Sci. 19 (2), 133–152, (2009).
	%%%%%%%
	\bibitem{Kim} H. Kim, The existence and uniqueness of very weak solutions of the stationary Boussinesq system, Nonlinear Anal. 75 (1), 317–330, (2012). 
	%%%%%%%
	\bibitem{PGLR1} P.G. Lemari\'e-Rieusset. \emph{The Navier-Stokes Problem in the 21st Century}. Chapman \& Hall/CRC, (2016).
	%%%%%%%
	\bibitem{Li} Q. Li, P. Wang, and X. Xu, Local well-posedness in Gevrey function spaces for 3D Boussinesq boundary layer system, J. Differential Equations 450, 113725, (2026).
	%%%%%%%
		\bibitem{Maimoto1}  H. Morimoto. \emph{On the existence of weak solutions of equation of natural convection}.  J. Fac. Sci. Univ. Tokyo 36: 87–102 (1989).
	%%%%%%%
	\bibitem{Marimoto2} H. Morimoto, \emph{On the existence and uniqueness of the stationary solution to the equations of natural convection}.  Tokyo J. Math. 14: 365–375 (1991). 
	%%%%%%%
	\bibitem{Naibo} V. Naibo and A. Thomson. \emph{Coifman–Meyer multipliers: Leibniz-type rules and applications to scattering of solutions to PDEs}. Trans. Amer. Math. Soc. 372, no. 8, 5453–5481 (2019).
	%%%%%%%
	\bibitem{Paicu} M. Paicu and V. Vicol. \emph{Analyticity and Gevrey-class regularity for the second-grade fluid
	equations}. J. Math. Fluid Mech., 13: 533-555 (2011). 
	%%%%%%%
	\bibitem{Pedlosky}  J. Pedlosky. \emph{Geophysical Fluid Dynamics}. 
	 2nd ed., Springer, (1987).
	%%%%%%%
	\bibitem{Seregin} G. Seregin. \emph{Liouville type theorem for stationary Navier–Stokes equations} Nonlinearity 29, 2191–2195, (2016).
	%%%%%%%
	\bibitem{Tan} W. Tan, New Liouville type theorems for the stationary Navier–Stokes equations, arXiv:2501.03609 (2025).
	%%%%%%%
	\bibitem{Vallis}  G. K. Vallis. \emph{Atmospheric and Oceanic Fluid Dynamics}. 2nd ed., Cambridge University Press, (2017).
	%%%%%%%
	\bibitem{Zhou} Y. Zhou. \emph{Analytical smoothing effect of solution for the Boussinesq equations}. Acta Math. Sin. (Engl. Ser.) 24 (11), 1777--1790, (2008).
	
	
	
\end{thebibliography}
\end{document}